\documentclass{amsart} 
\usepackage{a4wide}
\usepackage{amsmath,amsthm,amscd}
\usepackage{enumerate}
\usepackage{amssymb,amsfonts} 
\input xy
\xyoption{all}
\numberwithin{equation}{section}
\theoremstyle{definition}
\newtheorem{dfn}{Definition}[section]
\newtheorem{example}[dfn]{Example}
\newtheorem{rem}[dfn]{Remark}
\theoremstyle{plain}
\newtheorem{thm}[dfn]{Theorem}
\newtheorem{prp}[dfn]{Proposition}
\newtheorem{cor}[dfn]{Corollary}
\newtheorem{lem}[dfn]{Lemma}

\newtheorem*{thmB}{Theorem A}       %
\newtheorem*{thmC}{Theorem B}       %

\title{Multiparameter Twisted Weyl Algebras}
\author{Vyacheslav Futorny \and Jonas T. Hartwig}
\address{Instituto de Matem\'{a}tica e Estat\'{i}stica, Universidade de S\~ao Paulo, S\~ao Paulo, Brasil}
\email{futorny@ime.usp.br}
\address{Department of Mathematics, Stanford University, Stanford, CA, USA}
\email{jonas.hartwig@gmail.com}

\newcommand\al{\alpha}\newcommand\be{\beta} \newcommand\ga{\gamma}
\newcommand\ep{\varepsilon}
\newcommand\ze{\zeta}   
 \newcommand\la{\lambda}\newcommand\La{\Lambda}
  \newcommand\si{\sigma}\newcommand\ta{\tau}   
     \newcommand\Om{\Omega}

 \newcommand\mf{\mathfrak{m}}\newcommand\nf{\mathfrak{n}}

\newcommand\C{\mathbb{C}}\newcommand\Z{\mathbb{Z}}
\newcommand\K{\Bbbk}

\newcommand\Max{\mathrm{Max}}    \newcommand\Supp{\mathrm{Supp}}
\newcommand\Specm{\mathrm{Specm}}
\newcommand\Ann{\mathrm{Ann}} \newcommand\Stab{\mathrm{Stab}}

\newcommand{\TGWC}[4]{\mathcal{C}_{#4}({#1},{#2},{#3})}  
\newcommand{\TGWA}[4]{\mathcal{A}_{#4}({#1},{#2},{#3})}
\newcommand{\TGWI}[4]{\mathcal{I}_{#4}({#1},{#2},{#3})}
\newcommand{\MTWA}[5]{A_{#1}^{#2}({#3},{#4},{#5})}  

\newcommand{\PM}[2]{\mathrm{PM}_{#1}({#2})}

\renewcommand{\hat}{\widehat}
\renewcommand{\tilde}{\widetilde}

\newcommand{\TGW}[2]{\mathsf{TGW}_{#1}({#2})}

\newcommand{\Af}[1]{\mathcal{A}_{#1}}

\newcommand{\un}{\underline}

\DeclareMathOperator{\Aut}{Aut}

\date{}
\begin{document}
\maketitle
\begin{abstract}
We introduce a new family of twisted generalized Weyl algebras, called
multiparameter twisted Weyl algebras, for which we parametrize
all simple quotients of a certain kind. Both Jordan's simple localization
of the multiparameter quantized Weyl algebra and Hayashi's $q$-analog of the Weyl algebra
are special cases of this construction. We classify all simple weight
modules over any multiparameter twisted Weyl algebra.
Extending results by Benkart and Ondrus, we also describe all Whittaker pairs up to
isomorphism over a class of twisted generalized Weyl algebras which includes the
multiparameter twisted Weyl algebras.
\end{abstract}

\tableofcontents

\section{Introduction}


Let $R$ be an algebra, $\si_1,\ldots,\si_n$  commuting algebra automorphisms of $R$,
 $t_1,\ldots,t_n$  elements from the center of $R$, and
$\mu_{ij}$ an $n\times n$ matrix of invertible scalars. To these data one associates
a \emph{twisted generalized Weyl algebra} $\TGWA{R}{\si}{t}{\mu}$, an associative
$\Z^n$-graded algebra (see Section \ref{sec:tgwadef} for definition). These algebras were
introduced by Mazorchuk and Turowska \cite{MT99} and they are generalizations of the
much studied generalized Weyl algebras, defined independently by Bavula \cite{B},
 Jordan \cite{Jordan1993}, and Rosenberg \cite{Rb} (called there \emph{hyperbolic rings}).

Simple weight modules over twisted generalized Weyl algebras have been studied in \cite{MT99},\cite{MPT},\cite{Ha06}.
In \cite{MT02} the authors classified bounded and unbounded $\ast$-representations over
twisted generalized Weyl algebras.
Interesting examples of twisted generalized Weyl algebras were given in \cite{MPT}.
In \cite{Ha09} new examples of twisted generalized Weyl algebras were constructed
from symmetric Cartan matrices.

In this paper we define  a family of twisted generalized Weyl algebras, called \emph{multiparameter twisted Weyl algebras}. This definition was inspired by an unpublished note by Benkart \cite{Be} where \emph{multiparameter Weyl algebras} were introduced.  These algebras are a particular case of the algebras of our class.

\emph{Multiparameter quantized Weyl algebras} $A_n^{\bar q,\Lambda}$
were introduced in \cite{M} as a generalization of the quantized Weyl algebras obtained
by Pusz and Woronowicz \cite{PW} in the context of quantum group
covariant differential calculus. They are examples of twisted generalized Weyl algebras.
Contrary to the usual Weyl algebras the algebra, $A_n^{\bar q,\Lambda}$
is in general not simple, even for generic parameters.
Jordan \cite{J} found a certain natural simple localization of $A_n^{\bar q,\Lambda}$.

The first main theorem of the paper parametrizes simple quotients of multiparameter twisted Weyl algebras in terms of maximal ideals of certain Laurent polynomial rings.
Jordan's localization of $A_n^{\bar q,\Lambda}$ is an example in this family,
as well as Hayashi's $q$-deformed Weyl algebras \cite{H90}.

\begin{thmB}
Let $A=\MTWA{n}{k}{r}{s}{\La}$ be a multiparameter twisted Weyl algebra.
\begin{enumerate}[{\rm (a)}]
\item
The assignment
\begin{equation}
 \nf\mapsto A/ \langle \nf \rangle
\end{equation}
where $\langle \nf\rangle$ denotes the ideal in $A$ generated by $\nf$,
is a bijection between the set of maximal ideals in the invariant subring $R^{\Z^n}$
and the set of simple quotients of $A$ in which all $X_i,\, Y_i\,(i=1,\ldots,n)$ are regular.
 
\item For any $\nf\in\Specm(R^{\Z^n})$, the quotient $A/\langle \nf \rangle$ is isomorphic
to the twisted generalized Weyl algebra
 $\TGWA{R/R\nf}{\bar \si}{\bar t}{\mu}$,
 where $\bar\si_g(r+R\nf)=\si_g(r)+R\nf, \,  \forall g\in\Z^n,r\in R$
 and $\bar t_i=t_i+R\nf, \, \forall i$.

\item
$A/\langle \nf \rangle$ is a domain for all $\nf\in\Specm(R^{\Z^n})$ if and only if
$\Z^{2n} / G$ is torsion-free, where $G$ is the gradation group of $R^{\Z^n}$.
\end{enumerate}
\end{thmB}

The second main theorem of the paper gives the explicit relation between four twisted generalized Weyl
algebras, namely the multiparameter quantized Weyl algebra $A_n^{\bar q,\Lambda}$, Jordan's localization
$B_n^{\bar q, \Lambda}$, a specific multiparameter twisted Weyl algebra $\MTWA{n}{k}{r}{s}{\Lambda}$ that we define,
and a certain quotient $\hat{\MTWA{n}{k}{r}{s}{\Lambda}}$ of it which is simple and isomorphic to $B_n^{\bar q, \Lambda}$.
\begin{thmC} We have a commutative diagram in the category of $\Z^n$-graded algebras:
\[
\xymatrix@C=0.5cm@M=1ex{
\MTWA{n}{k}{r}{s}{\Lambda}  \ar@{->>}[d] \ar@{->>}[drr]
   &&                       \\
\hat{\MTWA{n}{k}{r}{s}{\Lambda}}
   \ar@<.5ex>@{->}[rr]^{\simeq}
   &&  B_n^{\bar q,\Lambda} \ar@<.5ex>@{->}[ll]   \\
                          &&  A_n^{\bar q,\Lambda} \ar@{^{(}->}[ull] \ar@{^{(}->}[u] 
}\]
\end{thmC}

We end the introduction with an overview of the content of this paper.
In Sections \ref{sec:finitistic} and \ref{sec:A1n}, we first consider 
certain families of twisted generalized Weyl algebras.
Section \ref{sec:mtwa} is devoted to the definition and structural results for multiparameter twisted Weyl algebras,
with a proof of Theorem B in Section \ref{sec:simplequotients}.
Examples and relations to existing algebras are given in Sections \ref{sec:mwa} and \ref{sec:jordan},
where Theorem C is proved.
Representations of multiparameter twisted Weyl algebras are studied in Sections \ref{sec:weight}
and \ref{sec:whittaker}.

\subsection*{Acknowledgements}
This work was carried out during the second author's postdoc at IME-USP,
funded by FAPESP, processo 2008/10688-1. 
The first  author is supported in part by the CNPq grant (301743/
2007-0) and by the Fapesp grant (2005/60337-2). 


\subsection*{Notation and conventions}
By ``ring'' (``algebra'') we mean unital associative ring (algebra).
All ring and algebra morphisms are required to be unital.
By ``ideal'' we mean two-sided ideal unless otherwise stated.
An element $x$ of a ring $R$ is said to be \emph{regular in $R$} if
for all nonzero $y\in R$ we have $xy\neq 0$ and $yx\neq 0$.
The set of invertible elements in a ring $R$ will be denoted by $R^\times$.

Let $R$ be a ring.
Recall that an $R$-ring is a ring $A$ together with a ring morphism $R\to A$.
Let $X$ be a set. Let $RXR$ be the free $R$-bimodule on $X$.
The free $R$-ring $F_R(X)$ on $X$ is defined as the tensor algebra of the free
$R$-bimodule on $X$: $F_R(X)=\oplus_{n\ge 0} (RXR)^{\otimes_R n}$
where $(RXR)^{\otimes_R 0}=R$ by convention and the ring morphism $R\to F_R(X)$
is the inclusion into the degree zero component.

\section{Twisted generalized Weyl algebras} \label{sec:tgwa}
Throughout this section we fix a commutative ring $\K$.
\subsection{Definition} \label{sec:tgwadef}
We recall the definition of twisted generalized Weyl algebras \cite{MT99,MPT}.
Here we emphasize the initial data more than usual, which will be
useful in the next section to express the functoriality of the construction.

\begin{dfn}[TGW datum]
Let $n$ a positive integer.
A \emph{twisted generalized Weyl datum (over $\K$ of degree $n$)}
 is a triple $(R,\si,t)$ where
\begin{itemize}
\item $R$ is a unital associative $\K$-algebra,
\item $\si$ is a group homomorphism $\si:\Z^n\to\Aut_\K(R)$, $g\mapsto \si_g$,
\item $t$ is a function $t:\{1,\ldots,n\}\to Z(R)$, $i\mapsto t_i$.
\end{itemize} 
A \emph{morphism} between TGW data over $\K$ of degree $n$, 
\[\varphi:(R,\si,t)\to (R',\si',t')\]
is a $\K$-algebra morphism $\varphi: R\to R'$ such that
$\varphi \si_i=\si_i'\varphi$ and $\varphi(t_i)=t_i'$ for all $i\in\{1,\ldots,n\}$.
We let $\TGW{n}{\K}$ denote the category whose objects are the TGW data
over $\K$ of degree $n$ and morphisms are as above.
\end{dfn}
For $i\in\{1,\ldots,n\}$ we put $\si_i=\si_{e_i}$, where $\{e_i\}_{i=1}^n$ is the standard
$\Z$-basis for $\Z^n$.
A \emph{parameter matrix (over $\K^\times$ of size $n$)} is an $n\times n$
matrix $\mu=(\mu_{ij})_{i\neq j}$ without diagonal where $\mu_{ij}\in\K^\times\,\forall i\neq j$. The set of all parameter matrices over $\K^\times$ of size $n$
will be denoted by $\PM{n}{\K}$.


\begin{dfn}[TGW construction]
Let $n\in\Z_{>0}$, $(R,\si,t)$ be an object in $\TGW{n}{\K}$, and $\mu\in\PM{n}{\K}$.
The \emph{twisted generalized Weyl construction with parameter matrix $\mu$
associated to the TGW datum $(R,\si,t)$} is denoted by $\TGWC{R}{\si}{t}{\mu}$
and is defined as the free $R$-ring on the set $\{x_i,y_i\mid i=1,\ldots,n\}$
modulo two-sided ideal generated by the following set of elements:
\begin{subequations}\label{eq:tgwarels}
\begin{alignat}{3}
\label{eq:tgwarels1}
x_ir  &-\si_i(r)x_i,  &\quad y_ir&-\si_i^{-1}(r)y_i, 
 &\quad \text{$\forall r\in R,\, i\in\{1,\ldots,n\}$,} \\
\label{eq:tgwarels2}
y_ix_i&-t_i, &\quad x_iy_i&-\si_i(t_i),
 &\quad \text{$\forall i\in\{1,\ldots,n\}$,} \\
\label{eq:tgwarels3}
&&\quad x_iy_j&-\mu_{ij}y_jx_i,
 &\quad \text{$\forall i,j\in\{1,\ldots,n\},\, i\neq j$.}
\end{alignat}
\end{subequations}
\end{dfn}
The images in $\TGWC{R}{\si}{t}{\mu}$ of the elements $x_i, y_i$  will
be denoted by $\hat X_i, \hat Y_i$ respectively.
The ring $\TGWC{R}{\si}{t}{\mu}$ has a $\Z^n$-gradation
given by requiring $\deg \hat X_i=e_i, \deg \hat Y_i=-e_i, \deg r=0\, \forall r\in R$.
Let $\TGWI{R}{\si}{t}{\mu}\subseteq \TGWC{R}{\si}{t}{\mu}$ be 
the sum of all graded ideals $J\subseteq \TGWC{R}{\si}{t}{\mu}$ having
zero intersection with the degree zero component, i.e. such that
$\TGWC{R}{\si}{t}{\mu}_0\cap J=\{0\}$.
It is easy to see that $\TGWI{R}{\si}{t}{\mu}$ is the unique maximal graded ideal having
zero intersection with the degree zero component.
\begin{dfn}[TGW algebra]
The \emph{twisted generalized Weyl algebra with parameter matrix $\mu$
associated to the TGW datum $(R,\si,t)$} is
denoted $\TGWA{R}{\si}{t}{\mu}$ and is defined as the
quotient $\TGWA{R}{\si}{t}{\mu}:=\TGWC{R}{\si}{t}{\mu} / \TGWI{R}{\si}{t}{\mu}$.
\end{dfn}
Since $\TGWI{R}{\si}{t}{\mu}$ is graded,
$\TGWA{R}{\si}{t}{\mu}$ inherits a $\Z^n$-gradation from $\TGWC{R}{\si}{t}{\mu}$.
The images in $\TGWA{R}{\si}{t}{\mu}$ of the elements $\hat X_i, \hat Y_i$ will be
denoted by $X_i, Y_i$.
By a \emph{monic monomial} in a TGW construction $\TGWC{R}{\si}{t}{\mu}$
(respectively TGW algebra $\TGWA{R}{\si}{t}{\mu}$)
we will mean a product of elements from
$\{\hat X_i, \hat Y_i\mid i=1,\ldots,n\}$
(respectively $\{X_i, Y_i\mid i=1,\ldots,n\}$).
 
The following statements are easy to check.
\begin{lem}\label{lem:easy}
\begin{enumerate}[{\rm (a)}]
\item \label{it:monomialgeneration}
$\TGWA{R}{\si}{t}{\mu}$ (respectively $\TGWC{R}{\si}{t}{\mu}$) is generated as a left and as a right $R$-module
by the monic monomials in $X_i, Y_i\,(i=1,\ldots,n)$ (respectively
$\hat X_i, \hat Y_i\, (i=1,\ldots,n)$).

\item
\label{it:A0}
The degree zero component of $\TGWA{R}{\si}{t}{\mu}$ is equal to
the image of $R$ under the natural map $\rho:R\to\TGWA{R}{\si}{t}{\mu}$.

\item\label{it:zerointersection}
Any nonzero graded ideal of $\TGWA{R}{\si}{t}{\mu}$ has nonzero intersection
with the degree zero component.
\end{enumerate}
\end{lem}

\begin{dfn}[$\mu$-Consistency]
Let $(R,\si,t)$ be a TGW datum over $\K$ of degree $n$ and
$\mu$ be a parameter matrix over $\K^\times$ of size $n$.
We say that $(R,\si,t)$ is \emph{$\mu$-consistent} if the
canonical map $\rho:R\to \TGWA{R}{\si}{t}{\mu}$ is injective.
\end{dfn}
Since $\TGWI{R}{\si}{t}{\mu}$ has zero intersection with the zero-component,
$(R,\si,t)$ is 
$\mu$-consistent iff the canonical map $R\to \TGWC{R}{\si}{t}{\mu}$ is 
injective.
Even in the cases when $\rho$ is not injective, we will often
view $\TGWA{R}{\si}{t}{\mu}$ as a left $R$-module and write for example
$rX_i$ instead of $\rho(r)X_i$.

\begin{dfn}[Regularity]
A TGW datum $(R,\si,t)$ is called \emph{regular} if $t_i$
is regular in $R$ for all $i$.
\end{dfn}

The following result was proved in \cite[Theorem 6.2]{FutHar2011}.
\begin{thm}\label{thm:muconsistency}
Let $\K$ be a commutative unital ring,
$R$ be an associative $\K$-algebra, $n$ a positive ingeter,
$t=(t_1,\ldots,t_n)$ be an $n$-tuple of regular central elements of $R$,
$\si:\Z^n\to\Aut_\K(R)$ a group homomorphism, $\mu_{ij}$ ($i,j=1,\ldots,n, i\neq j$) invertible
elements from $\K$, and
$\TGWA{R}{\si}{t}{\mu}$ the corresponding twisted generalized Weyl algebra,
equipped with the canonical homomorphism of $R$-rings
$\rho:R\to \TGWA{R}{\si}{t}{\mu}$.
Then the following two statements are equivalent:
\begin{enumerate}[{\rm (a)}]
\item $\rho$ is injective,
\item the following two sets of relations are satisfied in $R$:
\begin{alignat}{2}
\label{eq:muconsistency1}
\si_i\si_j(t_it_j) &= \mu_{ij}\mu_{ji}\si_i(t_i)\si_j(t_j),
  &\qquad \forall i,j=1,\ldots,n,\, i\neq j, \\
\label{eq:muconsistency2}
t_j\si_i\si_k(t_j) &= \si_i(t_j)\si_k(t_j),
  &\qquad \forall i,j,k=1,\ldots,n,\, i\neq j\neq k\neq i.
\end{alignat} 
 \end{enumerate}
In particular, if \eqref{eq:muconsistency1} and \eqref{eq:muconsistency2}
are satisfied, then $\TGWA{R}{\si}{t}{\mu}$ is nontrivial iff $R$ is nontrivial.
Moreover, neither of the two conditions \eqref{eq:muconsistency1} and
 \eqref{eq:muconsistency2} imply the other.
\end{thm}

\begin{lem} \label{lem:tgwatgwc}
If $t_i\in R^\times$ for all $i$, then
the canonical projection $\TGWC{R}{\si}{t}{\mu}\to\TGWA{R}{\si}{t}{\mu}$
is an isomorphism.
\end{lem}
\begin{proof}
The algebra $\TGWC{R}{\si}{t}{\mu}$
is a $\Z^n$-crossed product algebra over its degree zero subalgebra,
since each homogenous component 
contains an invertible element. Indeed
since $t_i\in R^\times$, each $X_i$ is invertible and thus
$X_1^{g_1}\cdots X_n^{g_n}$ has
degree $g$ and is invertible. Therefore any nonzero graded ideal in
$\TGWC{R}{\si}{t}{\mu}$ has nonzero intersection with the degree zero component, a
property which holds for any strongly graded ring, in particular
for crossed product algebras. Thus $\TGWI{R}{\si}{t}{\mu}=0$, which
proves the claim.
\end{proof}

\subsection{TGW algebras which are domains}\label{sec:tgwadomains}
The following condition for a TGW algebra to be a domain will be used.
\begin{prp}\label{prp:domain}
Let $A=\TGWA{R}{\si}{t}{\mu}$ be a twisted generalized Weyl algebra
where $(R,\si,t)$ is $\mu$-consistent.
Then $A$ is a domain if and only if $R$ is a domain.
\end{prp}
\begin{proof}
Clearly $R$ must be a domain if $A$ is a domain.
For the converse, assume $R$ is a domain and suppose
$a,b\in A$ are nonzero but $ab=0$. Let $a_g$ and $b_h$ be the
leading terms in $a$ and $b$ respectively, with respect to some order
(we use here that the group $\Z^n$ is orderable).
Then $a_gb_h=0$. As in the proof of \cite[Proposition 3.1]{Ha09} this
forces $a_g=0$ or $b_h=0$ which is a contradiction.
\end{proof}

\section{Finitistic TGW algebras} \label{sec:finitistic}
Throughout the rest of the paper we assume that $\K$ is a field.


In \cite{Ha09} the following notion (there called ''locally finite'' TGW algebra) was defined.
\begin{dfn}\label{dfn:locfin}
A TGW algebra $A=\TGWA{R}{\si}{t}{\mu}$ is called \emph{$\K$-finitistic} if
$\dim_\K V_{ij}<\infty$ for all $i,j\in\{1,\ldots,n\}$, where
\begin{equation}\label{eq:Vijdef}
V_{ij} = \mathrm{Span}_{\K}
 \{\si_i^k(t_j)\mid  k\in\Z\}.
\end{equation}
\end{dfn}
For each $i,j$, we denote by $p_{ij}\in\K[x]$ be the minimal
polynomial for $\si_i$ acting on the finite-dimensional
space $V_{ij}$.
The following result was proved in \cite{Ha09} (for the case
$\mu_{ij}=\mu_{ji}$ and $R$ a commutative domain, but these restrictions are unnecessary).
\begin{thm}\label{thm:lfthm}
Let $A=\TGWA{R}{\si}{t}{\mu}$ be a $\K$-finitistic TGW algebra
where $(R,\si,t)$ is $\mu$-consistent.
Let $p_{ij}$ be the minimal polynomials defined above.
\begin{enumerate}[{\rm (a)}]
\item Define the matrix $C_A=(a_{ij})$ with integer entries as follows:
\begin{equation}\label{eq:CAdef}
a_{ij}=
\begin{cases}
2,&i=j,\\
1-\deg p_{ij},&i\neq j.
\end{cases}
\end{equation}
Then $C_A$ is a generalized Cartan matrix.
\item Assume $t_i$ is regular in $R$ for all $i$. Writing
\[p_{ij}(x)=x^{m_{ij}}+\la_{ij}^{(1)}x^{m_{ij}-1}+\cdots+\la_{ij}^{(m_{ij})},\]
where all $\la_{ij}^{(k)}\in\K$, the following identities hold in $A$, for any $i\neq j$:
\begin{equation}\label{eq:genserre3}
 X_i^{m_{ij}}X_j+\la_{ij}^{(1)}\mu_{ij}^{-1} X_i^{m_{ij}-1}X_jX_i+
 \cdots + \la_{ij}^{(m_{ij})} \mu_{ij}^{-m_{ij}}X_jX_i^{m_{ij}}=0
\end{equation}
and
\begin{equation}\label{eq:genserre3Y}
 Y_jY_i^{m_{ij}}+\la_{ij}^{(1)}\mu_{ji}^{-1} Y_iY_jY_i^{m_{ij}-1}+
 \cdots + \la_{ij}^{(m_{ij})} \mu_{ji}^{-m_{ij}}Y_i^{m_{ij}}Y_j=0.
\end{equation}
Moreover, for any $i\neq j$ and $m<m_{ij}$, the sets
$\{X_i^{m-k}X_jX_i^k\}_{k=0}^m$ and
$\{Y_i^{m-k}Y_jY_i^k\}_{k=0}^m$ are linearly independent in $A$ over $\K$.
\end{enumerate}
\end{thm}
This gives an interpretation of the minimal polynomials $p_{ij}$ for
$i\neq j$ in terms of identities in the algebra $A$.
If $C_A$ is of type $Z$ (for example $A_n, B_n, C_n, D_n, E_6,E_7,E_8$ etc)
we say that $A$ is of Lie type $Z$.

Here we note that the polynomials $p_{ii}$ also give rise to
identities in $A$.
\begin{thm}\label{eq:lfthm2}
Let $A=\TGWA{R}{\si}{t}{\mu}$ be a $\K$-finitistic TGW algebra,
where $(R,\si,t)$ is regular and $\mu$-consistent.
Let $p_{ii}\in\K[x]$ be the minimal
polynomial for $\si_i$ acting on the finite-dimensional
spaces $V_{ii}$ defined in \eqref{eq:Vijdef}.
Writing
\[p_{ii}(x)=x^{m_{ii}}+\la_{ii}^{(1)}x^{m_{ii}-1}+\cdots+\la_{ii}^{(m_{ii})},\]
where all $\la_{ii}^{(k)}\in\K$, the following identities hold in $A$, for any $i$:
\begin{equation}\label{eq:genserre4}
 X_i^{m_{ii}}Y_i+\la_{ii}^{(1)} X_i^{m_{ii}-1}Y_iX_i+
 \cdots + \la_{ii}^{(m_{ii})} Y_iX_i^{m_{ii}}=0
\end{equation}
and
\begin{equation}\label{eq:genserre4Y}
 X_iY_i^{m_{ii}}+\la_{ii}^{(1)} Y_iX_iY_i^{m_{ii}-1}+
 \cdots + \la_{ii}^{(m_{ii})} Y_i^{m_{ii}}X_i=0.
\end{equation}
Moreover, for any $i$ and $m<m_{ii}$, the sets
$\{X_i^{m-k}Y_iX_i^k\}_{k=0}^m$ and
$\{Y_i^{m-k}X_iY_i^k\}_{k=0}^m$ are linearly independent in $A$ over $\K$.
\end{thm}
\begin{proof}
The proof is similar to the proof of Theorem \ref{thm:lfthm}.
\end{proof}

\begin{example}
Let $R=\K[t_1,\ldots,t_n]$ be the polynomial algebra,
$\si_i(t_j)=t_j-\delta_{ij}$,
$\mu_{ij}\in \K\backslash\{0\}$ such that $\mu_{ij}\mu_{ji}=1$.
Let $A=\TGWA{R}{\si}{t}{\mu}$ be the associated TGW algebra.
It is easy to see that it is $\K$-finitistic. For $i=j$, the minimal polynomials are
$p_{ii}(x)=(x-1)^2$. For $i\neq j$ we have
$p_{ij}(x)=x-1$. The matrix $C_A$, defined in \eqref{eq:CAdef}, 
is the Cartan matrix of type $(A_1)^n=A_1\times\cdots\times A_1$
(just a diagonal matrix with $2$ on the diagonal). Thus 
$A$ is of Lie type $(A_1)^n$.
By \eqref{eq:genserre3} we have $X_iX_j = \mu_{ij}^{-1}X_jX_i$
for $i\neq j$. If all $\mu_{ij}=1$ then $A$ is isomorphic to the
$n$:th Weyl algebra.
\end{example}

\begin{example}
The following TGW algebra was first mentioned as an example 
in \cite{MT99}, but a complete presentation by generators
and relations was only given in \cite{Ha09}.
Let 
 $n=2$, $R=\K[H]$, $\si_1(H)=H+1$, $\si_2(H)=H-1$,
 $t_1=H$, $t_2=H+1$, $\mu_{12}=\mu_{21}=1$ and
let $A=\TGWA{R}{\si}{t}{\mu}$ be the associated TGW algebra.
Clearly $A$ is locally finite with $V_{ij}=\C H\oplus \C 1$ for $i,j=1,2$.
Observing that $\si_2(t_1)$ and $t_1$ are linearly independent and that
\[\si_2^2(t_1)-2\si_2(t_1)+t_1=H-2 - 2(H-1) + H = 0\]
we see that the minimal polynomial $p_{21}$ for $\si_2$ acting on $V_{21}$
is given by $p_{21}(x)=x^2-2x+1=(x-1)^2$. Similarly one checks that in fact
$p_{ij}(x)=(x-1)^2$ for all $i,j=1,2$.
Thus $C_A=\left[\begin{smallmatrix}2&-1\\-1&2\end{smallmatrix}\right]$, the
Cartan matrix of type $A_2$, so $A$ is of Lie type $A_2$.
By Theorem \ref{thm:lfthm}b),
we have for example $X_1^2X_2-2X_1X_2X_2+X_2X_1^2=0$
in $A$, which is precisely one of the Serre relations in the enveloping algebra
of $\mathfrak{sl}_3(\K)$, the simple Lie algebra of type $A_2$.
It was shown in \cite[Example 6.3]{Ha09} that in fact $A$ is isomorphic to the $\K$-algebra
with generators $X_1,X_2,Y_1,Y_2,H$ and defining relations
\begin{gather*}
\begin{alignedat}{2}
X_1H&=(H+1)X_1, &\qquad  X_2H&=(H-1)X_2, \\
Y_1H&=(H-1)Y_1, &\qquad Y_2H&=(H+1)Y_2, \\
Y_1X_1&=X_2Y_2=H, &\qquad   Y_2X_2&=X_1Y_1=H+1,
\end{alignedat}\qquad
\begin{aligned}
X_1^2X_2-2X_1X_2X_1+X_2X_1^2&=0,\\
X_2^2X_1-2X_2X_1X_2+X_1X_2^2&=0,\\
Y_1^2Y_2-2Y_1Y_2Y_1+Y_2Y_1^2&=0,\\
Y_2^2Y_1-2Y_2Y_1Y_2+Y_1Y_2^2&=0.
\end{aligned}
\end{gather*}
In \cite{Ha09}, this TGW algebra was also generalized to arbitrary symmetric
generalized Cartan matrices, although explicit presentation was only given in
type $A_2$.
\end{example}

\section{TGW algebras of Lie type $(A_1)^n$} \label{sec:A1n}

\subsection{Presentation by generators and relations} \label{sec:A1n_pres}
Let $A=\TGWA{R}{\si}{t}{\mu}$ be a $\K$-finitistic TGW algebra of Lie type
 $(A_1)^n=A_1\times \cdots\times A_1$, with $(R,\si,t)$ being $\mu$-consistent.
Thus $C_A$ has all zeros outside the main diagonal.
That is, $\deg p_{ij} = 1$ for all $i\neq j$. Equivalently (since $p_{ij}$
are monic by definition), for $i\neq j$ we have
 $p_{ij}(x)=x-\ga_{ij}$ for some $\ga_{ij}\in\K\backslash\{0\}$.
By Theorem \ref{thm:lfthm},
this means that in $A$ we have
\begin{subequations}\label{eq:kfinrels}
\begin{align}
X_iX_j &= \ga_{ij}\mu_{ij}^{-1}X_jX_i\qquad\forall i\neq j,\\
Y_jY_i &= \ga_{ij}\mu_{ji}^{-1}Y_iY_j\qquad\forall i\neq j.
\end{align}
\end{subequations}
It also means that 
\begin{equation}\label{eq:A1nsigmagamma}
\si_i(t_j)=\ga_{ij}t_j \quad\text{for all $i\neq j$.}
\end{equation}
By Theorem \ref{thm:muconsistency}, $(R,\si,t)$ is $\mu$-consistent if and only if
\begin{equation}\label{eq:A1nconsistency}
\mu_{ij}\mu_{ji}=\ga_{ij}\ga_{ji}\quad\forall i\neq j.
\end{equation}
We can now prove that \eqref{eq:kfinrels} generate all relations in
the ideal $\TGWI{R}{\si}{t}{\mu}$. 
\begin{thm}\label{thm:A1n_pres}
Let $A=\TGWA{R}{\si}{t}{\mu}$ is a $\K$-finitistic TGW algebra of type $(A_1)^n$,
where $(R,\si,t)$ is regular and $\mu$-consistent (i.e. \eqref{eq:A1nconsistency} holds).
Then $A$ is isomorphic
to the $R$-ring generated by $X_1,\ldots,X_n,Y_1,\ldots,Y_n$ modulo the relations
\begin{gather}
\begin{alignedat}{3}
X_ir &=\si_i(r)X_i, &\quad Y_ir &=\si_i^{-1}(r)Y_i
 &&\quad \forall r\in R,\forall i, \\
Y_iX_i &=t_i, &\quad X_iY_i&=\si_i(t_i),
 &&\quad \forall i, 
\end{alignedat}\\
\begin{aligned} \label{eq:pres_rels2}
X_iY_j = \mu_{ij}Y_jX_i, \quad
X_iX_j = \ga_{ij}\mu_{ij}^{-1}X_jX_i, \quad
Y_jY_i = \ga_{ij}\mu_{ji}^{-1}Y_iY_j, \quad\forall i\neq j.
\end{aligned}
\end{gather}
\end{thm}
\begin{proof}
The statement is equivalent to that the ideal $I:=\TGWI{R}{\si}{t}{\mu}$
of the TGW construction $A'=\TGWC{R}{\si}{t}{\mu}$ used to
construct $A$ is generated by the set
\begin{equation}\label{eq:gens}
\big\{X_iX_j-\ga_{ij}\mu_{ij}^{-1}X_jX_i,\;
   Y_jY_i-\ga_{ij}\mu_{ji}^{-1}Y_iY_j\mid i\neq j\big\}.
\end{equation}
Let $J$ denote the ideal in $A'$ generated by \eqref{eq:gens}.
Assume $a\in I$. We will prove that $a\in J$.
Since $I$ is graded, we can without loss of generality
assume that $a$ is homogenous. Let $g=\deg a\in \Z^n$.
Without loss of generality we can also add to $a$ any element of $J$, and thus
can rearrange the $Y$'s and the $X$'s in the terms in $a$ so to obtain
\[a=rZ_1^{(g_1)}\cdots Z_n^{(g_n)}\]
for some $r\in R$, where $Z_i^{(k)}$ equals $X_i^k$ if $k\ge 0$ and $Y_i^{-k}$ otherwise.
Put $b=Z_n^{(-g_n)}\cdots Z_1^{(-g_1)}$. Then $ab$ has degree
zero so that $ab\in I\cap R=\{0\}$, forcing $ab=0$. On the other hand,
using \eqref{eq:tgwarels1},\eqref{eq:tgwarels2} repeatedly we have
$ab=rs$ where $s\in R$ is a product of elements of the form $\si_g(t_i)$
($g\in\Z^n$, $i\in\{1,\ldots,n\}$).
Since the $t_i$ are assumed to be regular in $R$ we must have $r=0$, i.e. $a=0\in J$.
\end{proof}

\subsection{Centralizer intersections and a simplicity criterion}
Let $G$ be a group acting by automorphisms on a ring $R$.
An ideal $J$ of $R$ is called $G$-invariant if $g(J)\subseteq J$ for
all $g\in G$. The ring $R$ is called $G$-simple if the only $G$-invariant
ideals are $R$ and $\{0\}$.
If $A\subseteq B$ are rings, then we let $C_B(A)$ denote the
centralizer of $A$ in $B$: $C_B(A)=\{b\in B\mid ba=ab\,\forall a\in A\}$.
In \cite[Th.~3.6, Th.~7.18]{HO} the following statements are proved.
\begin{thm}\label{thm:oinert}
\begin{enumerate}[{\rm (a)}]
\item \label{it:intersection}
Let $A=\TGWA{R}{\si}{t}{\mu}$ be a TGW algebra where $R$ is commutative
and $(R,\si,t)$ is $\mu$-consistent.
 Let $J$ be any nonzero ideal of $A$. Then $J\cap C_A(R)\neq 0$.
\item\label{it:tgwasimplicity2}
Let $A=\TGWA{R}{\si}{t}{\mu}$ be a TGW algebra,
and assume $(R,\si,t)$ is $\mu$-consistent.
Suppose $A$ is $\K$-finitistic of Lie type $(A_1)^n$.
Then $A$ is simple if and only if (1) $R$ is $\Z^n$-simple, (2) $Z(A)\subseteq R$, and (3) $Rt_i+R\si_i^d(t_i)=R$
for all $d\in\Z_{>0}$ and $i=1,\ldots,n$.
\end{enumerate}
\end{thm}
Part (b) is proved in \cite[Th.~7.18]{HO} in the more general context of so called $R$-finitistic TGW algebras. The result is a generalization of D. Jordan's simplicity criterion for generalized Weyl algebras \cite[Theorem 6.1]{Jordan1993}.

\section{Multiparameter twisted Weyl algebras} \label{sec:mtwa}

Now we define a special class of twisted generalized Weyl algebras. The definition of these algebras was inspired by 
a class of \emph{multiparameter Weyl algebras} introduced by Benkart \cite{Be}.  

\subsection{Definition} \label{sec:MTWA_def}
Let $n\in\Z_{>0}$ and $k\in\Z\backslash\{0\}$ and let
 $\Lambda=(\lambda_{ij})$, $r=(r_{ij})$ and $s=(s_{ij})$ be three
   $n\times n$-matrix with entries from $\K\backslash\{0\}$, such that
\begin{subequations}\label{eq:mtwa_conds}
\begin{align}
&\text{ $\la_{ii}=1\; \forall i$ and $\la_{ij}\la_{ji}=1\; \forall i\neq j$, }\label{eq:lacond}\\
&\text{ $r_{ii}/s_{ii}$ is a nonroot of unity $\forall i$, } \label{eq:rscond1}\\
&\text{ $r_{ij}^k=s_{ij}^k\;\;\forall i\neq j$.} \label{eq:rscond2}
\end{align}
\end{subequations}
Let
\begin{equation}
R=\K[u_1^{\pm 1},\ldots,u_n^{\pm 1}, v_1^{\pm 1},\ldots,v_n^{\pm 1}]
\end{equation}
be the Laurent polynomial ring over $\K$ in $2n$ indeterminates,
define $\si_1,\ldots,\si_n\in\Aut_\K(R)$ by
\begin{equation}\label{eq:laurentsigmas}
\si_i(u_j)=r_{ij}^{-1}u_j,\qquad \si_i(v_j)=s_{ij}^{-1}v_j,
\end{equation}
for all $i,j\in\{1,\ldots,n\}$, and define $t_1,\ldots,t_n\in R$ by
\begin{equation}\label{eq:mtwa_ti_def}
t_i=\frac{(r_{ii}u_i)^k - (s_{ii}v_i)^k}{r_{ii}^k-s_{ii}^k}.
\end{equation}
Finally, put 
\begin{equation}
\mu_{ij}=r_{ji}^{-k}\la_{ji}
\end{equation}
for all $i\neq j$.
Then one easily checks that the consistency relations \eqref{eq:muconsistency1},
\eqref{eq:muconsistency2} hold.
Thus, by Theorem \ref{thm:muconsistency}, the TGW datum $(R,\si,t)$ is $\mu$-consistent,
that is, the natural map $\rho:R\to\TGWA{R}{\si}{t}{\mu}$ is injective.
We denote the TGW algebra $\TGWA{R}{\si}{t}{\mu}$ by $\MTWA{n}{k}{r}{s}{\La}$ and call it a
\emph{multiparameter twisted Weyl algebra}. It is easy to see that it is $\K$-finitistic of
Lie type $(A_1)^n$ and
thus, by Theorem \ref{thm:A1n_pres}, $\MTWA{n}{k}{r}{s}{\La}$ is isomorphic to the unital associative $\K$-algebra
generated by $u_i^{\pm 1},\, v_i^{\pm 1},\, X_i,\, Y_i\,(i=1,\ldots,n)$ modulo the relations
\begin{subequations}\label{eq:mtwarels}
\begin{gather}
\text{the $u_i^{\pm 1}, v_j^{\pm 1}$ all commute and $u_iu_i^{-1}=v_iv_i^{-1}=1\,\forall i$,}\\
X_iX_j=\big(\frac{r_{ji}}{r_{ij}}\big)^k\la_{ij} X_jX_i  \quad\forall i,j,\\
Y_iY_j=\la_{ij}Y_jY_i \quad\forall i,j,\\
X_iY_j=r_{ji}^{-k}\la_{ji}Y_jX_i \quad\forall i\neq j,\\
Y_iX_i=\frac{(r_{ii}u_i)^k - (s_{ii}v_i)^k}{r_{ii}^k-s_{ii}^k}, \qquad
X_iY_i=\frac{u_i^k - v_i^k}{r_{ii}^k-s_{ii}^k} \quad\forall i,\\
X_iu_j=r_{ij}^{-1} u_jX_i,\quad X_iv_j=s_{ij}^{-1} v_jX_i,\quad 
Y_iu_j=r_{ij}u_jY_i,\quad Y_iv_j=s_{ij} v_jY_i\quad\forall i,j.
\end{gather}
\end{subequations}

\begin{rem}
One can also consider the larger class of algebras in which \eqref{eq:rscond1}
in the definition of $\MTWA{n}{k}{r}{s}{\La}$ is replaced by the weaker condition that 
$r_{ii}^k\neq s_{ii}^k$ for all $i$. However in this paper
we will always assume \eqref{eq:rscond1}, which in examples corresponds
to that ``$q$ is not a root of unity''.
\end{rem}

\subsection{Properties of multiparameter twisted Weyl algebras}
Let $R^{\Z^n}=\{r\in R\mid \si_i(r)=r\;\forall i=1,\ldots,n\}$ be the
invariant subring of $R$ under $\Z^n$.
For $d\in\Z^{2n}$, put $u^d=u_1^{d_1}\cdots u_n^{d_n}v_1^{d_{n+1}}\cdots v_n^{d_{2n}}$.
Let 
\begin{equation}\label{eq:Gdef}
G=\{d\in\Z^{2n}\mid u^d \in R^{\Z^n}\}.
\end{equation}
We have $R^{\Z^n}=\bigoplus_{d\in G} \K u^d$.

\begin{prp}\label{prp:maxcomm}
\begin{enumerate}[{\rm (a)}]
\item If $J$ is a proper $\Z^n$-invariant ideal of $R$,
then the group homomorphism $\Z^n\to\Aut_\K(R/J)$, induced by
the $\Z^n$-action on $R$, is injective.
\item If $J$ is a proper $\Z^n$-invariant ideal of $R$,
such that $R/J$ is $\Z^n$-simple or a domain, then
$R/J$ is maximal commutative in $\bar A:=\TGWA{R/J}{\bar\si}{\bar t}{\mu}$.
\end{enumerate}
\end{prp}
\begin{proof}
(a) Assume $g=(g_1,\ldots, g_n)\in\Z^n$ is such that $\si_g(p+J)=p+J$
for all $p+J\in R/J$.
Suppose that $g_i\neq 0$ for some $i$.
Then, taking $p=u_i$ we have $u_i+J=\si_{kg}(u_i)+J=
r_{1i}^{kg_1}\cdots r_{ni}^{kg_n} u_i+J$,
giving $(r_{1i}^{kg_1}\cdots r_{ni}^{kg_n}-1)u_i\in J$.
Since $J$ is proper and $u_i$ invertible we must have
$r_{1i}^{kg_1}\cdots r_{ni}^{kg_n}=1$.
Similarly taking $p=v_i$
gives that $s_{1i}^{kg_1}\cdots s_{ni}^{kg_n}=1$. But $r_{ij}^k=s_{ij}^k$ for $i\neq j$
and thus we get $r_{ii}^{kg_i}/s_{ii}^{kg_i}=1$, contradicting the fact that
$r_{ii}/s_{ii}$ is not a root of unity. Thus $g_i=0$ for all $i$.

(b) Follows from part (a) and \cite[Corollary 5.2]{HO}.
\end{proof}

\begin{prp}\label{prp:ideals_are_graded}
Any ideal of $A$ is graded.
\end{prp}
\begin{proof}
Let $J$ be any ideal in $A$ and let $a\in J$. Write $a=\sum_{g\in\Z^n}a_g$, where
$a_g\in A_g$ for each $g$. Pick any $h\in\Z^n$. We will show that $a_h\in J$.
By Proposition \ref{prp:maxcomm}a), the group morphism $\Z^n\to\Aut_\K(R)$ is injective.
So if $g\in\Z^n$, $g\neq h$, then there is a $d\in\Z^{2n}$ such that $\si_h(u^d)\neq \si_g(u^d)$.
By definition of the automorphisms $\si_i$ we have
 $\si_h(u^d)=\xi_h u^d$ and $\si_g(u^d)=\xi_g u^d$, for some nonzero $\xi_g,\xi_h\in\K$.
Put $b=\xi_g a - u^{-d} a u^d$. Then $b\in J$ and writing $b=\sum_{f\in\Z^n} b_f$
where $b_f\in A_f$ we have $b_g = \xi_g a_g- u^{-d} a_g u^d = (\xi_g-u^{-d}\si_g(u^d))a_g=0$,
and $b_h = \xi_g a_h - u^{-d}a_h u^d = (\xi_g-\xi_h)a_h$.
So, replacing $a$ by $(\xi_g-\xi_h)^{-1}b$, we have an element in $J$ with the same degree $h$
component but with the $g$ component eliminated. Repeating this we can eliminate
all components except $a_h$ and thus we obtain that $a_h\in J$.
\end{proof}

\begin{prp}\label{prp:ideal_bijection}
Let $\mathfrak{I}(R^{\Z^n})$ denote the set of ideals of $R^{\Z^n}$
and $\mathfrak{I}(R)^{\Z^n}$ denote the set of $\Z^n$-invariant ideals of $R$.
Consider the maps
\begin{align*}
\varepsilon: &\mathfrak{I}(R^{\Z^n})\to \mathfrak{I}(R)^{\Z^n},\quad \nf\mapsto R\nf, \\
\rho:  &\mathfrak{I}(R)^{\Z^n} \to \mathfrak{I}(R^{\Z^n}), \quad J\mapsto R^{\Z^n}\cap J.
\end{align*}
Then $\varepsilon$ and $\rho$ are inverse to eachother and set up an order-preserving
bijection between the two sets.
In particular, for each $\nf\in\Specm(R^{\Z^n})$,
 $R\nf$ is maximal among $\Z^n$-invariant ideals of $R$ and conversely,
 every maximal $\Z^n$-invariant ideal of $R$ equals $R\nf$ for some $\nf\in\Specm(R^{\Z^n})$.
\end{prp}
\begin{proof}
We can view $R$ as a module over $\K[\Z^n]$ by linearly extending the $\Z^n$-action on $R$.
Let $\K[\Z^n]^\ast$ be the group of characters (i.e. algebra morphisms $\K[\Z^n]\to \K$).
The product is given by $\chi_1\chi_2(g)=\chi_1(g)\chi_2(g)$ for $\chi_1,\chi_2\in\K[\Z^n]^\ast, g\in\Z^n$.
By definition of $\si_i$,
there is for each $d\in\Z^{2n}$ a $\chi\in\K[\Z^n]^\ast$ such that $\si_g(u^d)=\chi(g)u^d$ for all $g\in\Z^n$.
Thus
\begin{equation}\label{eq:Rsemisimple}
R=\bigoplus_{\chi\in\K[\Z^n]^\ast} R[\chi],\quad R[\chi]=\{r\in R\mid a.r=\chi(a)r\,\forall a\in\K[\Z^n]\}.
\end{equation}
In particular $R$ is semisimple as a module over $\K[\Z^n]$.
Using that each $R[\chi]$ is spanned by certain $u^d$, one verifies that
the decomposition \eqref{eq:Rsemisimple} turns $R$ into a strongly $\K[\Z^n]^\ast$-graded ring,
that is, $R[\chi_1]R[\chi_2]=R[\chi_1\chi_2]$ for all $\chi_1,\chi_2\in\K[\Z^n]^\ast$.
Moreover $R^{\Z^n}=R[\boldsymbol{1}]$ where $\boldsymbol{1}$ is the unit in
the character group $\K[\Z^n]^\ast$, given by $\boldsymbol{1}(g)=1$ for all $g\in\Z^n$.

We now show the maps $\varepsilon$ and $\rho$ are inverses to eachother.
Let $J$ be any $\Z^n$-invariant ideal of $R$. Thus it is a $\K[\Z^n]$-submodule of $R$,
and therefore $J=\bigoplus_{\chi\in\K[\Z^n]^\ast} J[\chi]$, where
$J[\chi]=R[\chi]\cap J$. Using the strong gradation property we have
$J[\chi]=R[\chi]R[\chi^{-1}]J[\chi]\subseteq R[\chi]J[\boldsymbol{1}]\subseteq J[\chi]$
which proves that $J[\chi]=R[\chi]J[\boldsymbol{1}]$ for all $\chi$.
Thus $J=R(J\cap R^{\Z^n})$. This shows that $\varepsilon\rho$ is the identity.
Let $\nf$ be an ideal of $R^{\Z^n}$.
Then $R\nf = \sum_{\chi\in\K[\Z^n]^\ast} R[\chi]\nf $ and $R[\chi]\nf\subseteq R[\chi]$.
Thus $(R\nf)\cap R^{\Z^n}=\nf$. This proves $\rho\varepsilon$ is the identity.


\end{proof}

We end this section with some lemmas that will be used in the next section
to prove the second main theorem of the paper.
\begin{lem}\label{lem:Jid}
 $Rt_i+R\si_i^d(t_i)=R$ for all $i=1,\ldots,n$ and all $d\in\Z_{>0}$.
\end{lem}
\begin{proof}
 We have
\begin{equation}
r_{ii}^{-dk}t_i - \si_i^d(t_i) =
 \frac{(-r_{ii}^{-dk}s_{ii}^k + s_{ii}^{k-dk})v_{ii}^k}{r_{ii}^k - s_{ii}^k}
\end{equation}
which is invertible since $r_{ii}/s_{ii}$ is assumed to not be a root of $1$
and since $v_i$ is invertible. This proves the claim.
\end{proof}

\begin{lem}\label{lem:products}
No product of elements of the form $\si_g(t_i)$ ($g\in\Z^n$, $i=1,\ldots,n$)
can belong to a $\Z^n$-invariant proper ideal of $R$.
\end{lem}
\begin{proof}
Indeed, such a product can be written $a=\xi \si_1^{p_1}(t_1)\cdots \si_n^{p_n}(t_n)$
for some nonzero $\xi\in\K$ and some $p_i\in\Z$. But then the proper
$\Z^n$-invariant ideal $L$ containing such element would also contain
$\si_1(a)$. By Lemma \ref{lem:Jid}, $Rt_1+R\si_1(t_1)=R$.
So for suitable $r_1,r_2\in R$,
$r_1a+r_2\si_1(a)=\xi' \si_2^{p_2}(t_2)\cdots\si_n^{p_n}(t_n)$ for some nonzero $\xi'\in\K$.
Continuing this way we would obtain that $L$ contains a nonzero scalar hence the $L=R$
contradicting that $L$ was proper. 
\end{proof}

\begin{lem}\label{lem:regularity}
Assume $J$ is a maximal $\Z^n$-invariant ideal of $R$.
Then all $t_i+J$ ($i=1,\ldots,n$) are regular in $R/J$.
\end{lem}
\begin{proof}
Let $T$ denote the multiplicative submonoid of $R/J$
generated by $\si_g(t_i+J)$ for all $g\in\Z^n$ and $i=1,\ldots,n$.
By Lemma \ref{lem:products}, $0\notin T$.
Observe that the set
\[L=\{\bar r\in R/J \mid u \text{$\bar r=0$ for some $u\in T$} \}\]
is a $\Z^n$-invariant ideal in $R/J$. But $L=R/J$ is impossible since
the ring $R/J$ is unital and $0\notin T$. Therefore, since
$R/J$ is $\Z^n$-simple, we have $L=0$ which proves that
in particular all $t_i+J$ ($i=1,\ldots,n$) are regular in $R/J$.
\end{proof}

\subsection{Simple quotients}\label{sec:simplequotients}
We come now to the main result  on the structure theory
of multiparameter twisted Weyl algebras. The following theorem (which is Theorem~B from the Introduction) describes
all quotients $A/Q$ of $A=\MTWA{n}{k}{r}{s}{\La}$ such that $A/Q$ is a
simple ring and such that the images of $X_i, Y_i$ in $A/Q$ are regular
for all $i$. It also gives a necessary and sufficient condition under
which all such quotients are domains.

We would like to emphasize that the subring $R^{\Z^n}$ of invariants
of $R$ under $\Z$ is just a Laurent polynomial ring over the field $\K$.
Thus there are plenty of explicitly known maximal ideals. Moreover, when
$\K$ is algebraically closed there is a bijection $\Specm(R^{\Z^n})\to (\K\backslash\{0\})^m$
where $m$ is the number of variables in $R^{\Z^n}$, i.e. the rank of the
subgroup $G\subseteq \Z^{2n}$ (see \eqref{eq:Gdef}). It is in this sense
we view the following theorem as a parametrization of the stated family of
simple quotients.

\begin{thm} \label{thm:family_of_simples}
Let $A=\MTWA{n}{k}{r}{s}{\La}$ be a multiparameter twisted Weyl algebra.
\begin{enumerate}[{\rm (a)}]
\item
The assignment
\begin{equation}
 \nf\mapsto A/ \langle \nf \rangle
\end{equation}
where $\langle \nf\rangle$ denotes the ideal in $A$ generated by $\nf$,
is a bijection between the set of maximal ideals in $R^{\Z^n}$
and the set of simple quotients of $A$ in which all $X_i,\, Y_i\,(i=1,\ldots,n)$ are regular.
 
\item For any $\nf\in\Specm(R^{\Z^n})$, the quotient $A/\langle \nf \rangle$ is isomorphic
to the twisted generalized Weyl algebra
 $\TGWA{R/R\nf}{\bar \si}{\bar t}{\mu}$,
 where $\bar\si_g(r+R\nf)=\si_g(r)+R\nf\,\forall g\in\Z^n,r\in R$
 and $\bar t_i=t_i+R\nf\,\forall i$.

\item
$A/\langle \nf \rangle$ is a domain for all $\nf\in\Specm(R^{\Z^n})$ if and only if
$\Z^{2n} / G$ is torsion-free, where $G$ was defined in \eqref{eq:Gdef}.
\end{enumerate}
\end{thm}
\begin{proof}
We first prove part (b). Let $\nf\in\Specm(R^{\Z^n})$. Put $J=R\nf$.
Trivially $AJA=\langle \nf \rangle$.
By Lemma \ref{lem:regularity}, $t_i+J$ are regular in $R/J$.
For each $g\in\Z^n$, $A_g=RZ_1^{(g_1)}\cdots Z_n^{(g_n)}$, $Z_i^{(m)}=X_i^m$ if $m\ge 0$
and $Z_i^{(m)}=Y_i^{-m}$ if $m<0$. 
We know $(R,\si,t)$ is $\mu$-consistent (see Section \ref{sec:MTWA_def}).
Thus by \cite[Cor.~6.4]{FutHar2011}, 
 $(R/J,\bar\si,\bar t)$ is
also $\mu$-consistent.
 Thus the claim follows from \cite[Thm.~4.1]{FutHar2011}
using \cite[Rem.~4.2]{FutHar2011}.

Now we prove part (a). Let $\nf\in\Specm(R^{\Z^n})$. By part (b),
$A/\langle\nf\rangle$ isomorphic to $\bar A:=\TGWA{R/J}{\bar \si}{\bar t}{\mu}$.
By Proposition \ref{prp:ideal_bijection}, $J$ is maximal among $\Z^n$-invariant ideals
of $R$, hence $R/J$ is $\Z^n$-simple. 
By Proposition \ref{prp:maxcomm}(b), $R/J$ is maximal commutative in $\bar A$. Hence, in particular, $Z(\bar A)\subseteq R/J$.
Let $\pi:R\to R/J$ be the canonical
projection. We have
$(R/J)\bar t_i+(R/J)\bar\si_i^d(\bar t_i) = \pi(Rt_i+R\si_i^d(t_i)) + J =R/J$
for any $i\in\{1,\ldots,n\}$ and $d\in\Z_{>0}$.
Thus the requirements in Theorem \ref{thm:oinert}(\ref{it:tgwasimplicity2}) are fulfilled
and we conclude that $\bar A$ is simple.
For each $i\in\{1,\ldots,n\}$, the elements $\bar t_i$, hence also $\bar\si_i(\bar t_i)$,
are regular in $R/J$. By the proof of 
\cite[Thm.~5.2(a)]{FutHar2011}, 
 these
elements are also regular in $\bar A$. Since $\bar t_i=Y_iX_i$  and
$\bar\si_i(\bar t_i)=X_iY_i$, it follows that $X_i$ and $Y_i$ are regular in $\bar A$.

Conversely, assume that $Q$ is any nonzero ideal of $A$ such that $A/Q$ is simple
and such that $X_i+Q$, $Y_i+Q$ are regular in $A/Q$ for all $i$.
By Proposition \ref{prp:maxcomm}, $R$ is maximal commutative in $A$, that is $C_A(R)=R$.
So by Theorem \ref{thm:oinert}(\ref{it:intersection}), $R\cap Q\neq 0$.
We claim that $R\cap Q$ is $\Z^n$-invariant. Let $i\in\{1,\ldots,n\}$ and
$p\in R\cap Q$. Then $X_i p \in Q$ since $Q$ is an ideal.
On the other hand, $X_i p =\si_i(p)X_i$.
Since the image of $X_i$ in $A/Q$ not a zero-divisor we conclude that $\si_i(p)\in Q$.
Trivially $\si_i(p)\in R$. Thus $\si_i(R\cap Q)\subseteq R\cap Q$ for all $i$.
Analogously one proves that $\si_i^{-1}(R\cap Q)\subseteq R\cap Q$ (or one can use that
$R$ is Noetherian). So $R\cap Q$ is indeed $\Z^n$-invariant.
Next we show that $R\cap Q$ is maximal among $\Z^n$-invariant ideals in $R$.
Suppose $R\cap Q\subsetneq J\subseteq R$ where $J$ is a $\Z^n$-invariant ideal of $R$.
Since $J$ is $\Z^n$-invariant, $AJ$ is a two-sided ideal of $A$.
Any element of $AJ + Q$ of degree zero has the form $p + a$ where $p\in J$ and 
$a$ is the degree zero component of an element of $Q$. But $Q$ is graded
by Proposition \ref{prp:ideals_are_graded} so $a\in Q$. Thus $(AJ+Q)\cap R=J+(Q\cap R)=J$.
Thus $AJ+Q$ is an ideal of $A$ which properly contains $Q$. Since $Q$ was 
maximal, $AJ+Q=A$ and thus $J=(AJ+Q)\cap R=R$. This shows that $R\cap Q$
is maximal among all $\Z^n$-invarant ideals of $R$.
By Proposition \ref{prp:ideal_bijection}, we conclude that $R\cap Q$ equals
$R\nf$ for some maximal ideal $\nf$ of $R^{\Z^n}$.
So for this $\nf$ we have $\langle \nf\rangle \subseteq Q$. But we proved
above that $A/\langle \nf\rangle$ is always simple. Thus $\langle\nf\rangle$
is a maximal ideal of $A$ which implies that $\langle\nf\rangle=Q$.

Finally,
two different ideals $\nf, \nf'$ in $R^{\Z^n}$
cannot generate the same maximal ideal $L$ in $A$,
since then $1\in \nf+\nf'\subseteq L$ which is absurd.

(c) By Proposition \ref{prp:domain} and part (b) we have that $A/\langle\nf\rangle$ is a domain iff
$R\nf$ is a prime ideal of $R$.
 Assume $R\nf$ is prime for all $\nf\in\Specm(R^{\Z^n})$.
Suppose $d\in \Z^{2n}$, $d\notin G$ but that there is a $p\in \Z_{>0}$ such that $pd\in G$.
Without loss of generality we can assume $p$ is prime.
Then there is a $j\in\{1,\ldots,n\}$ such that $\si_j(u^d)=\zeta u^d$ where $\zeta\in\K$, $\zeta\neq 1$, $\zeta^p=1$.
 Pick any $\Z$-basis
$\{d_1,\ldots, d_N\}$ for $G$ and take $\nf$ to be the maximal ideal in $R^{\Z^n}$ generated
by $u^{d_i}-1$ for $i=1,\ldots,N$. Then $u^{pd}-1\in \nf$ also, because $pd$ is a $\Z$-linear
combination of the $d_i$. But $u^{pd}-1=(u^d-1)(u^d-\zeta)\cdots (u^d-\zeta^{p-1})$.
Since $R\nf$ is prime we conclude that $u^d-\zeta^e\in R\nf$ for some $e\in\{0,\ldots,p-1\}$.
 However $R\nf$ is $\Z^n$-invariant and thus
$R\nf\ni u^d-\zeta^e - \ze^{-1}\si_j(u^d-\zeta^e) =(\ze^{-1}-1)\zeta^e$ which is invertible.
This contradicts that $R\nf$ is a proper ideal of $R$ which we know by
Proposition \ref{prp:ideal_bijection}.
Hence $\Z^{2n}/G$ is torsion-free.

Conversely, assume that $\Z^{2n}/G$ is torsion-free. Thus $\Z^{2n}\simeq G\oplus G'$ for
some subgroup $G'$ of $\Z^{2n}$. 
Therefore, viewing $R$ as the group algebra $\K[\Z^{2n}]$, we have an isomorphism
$R=\K[\Z^{2n}]\simeq \K[G]\otimes_\K \K[G']$. Under this isomorphism, $R \nf$
(where $\nf\in \Specm(R^{\Z^n})$ is arbitrary)
 is mapped to $\nf\otimes \K[G']$
which is a prime ideal in $\K[G]\otimes_\K \K[G']$ since
\[\frac{\K[G]\otimes_\K \K[G']}{\nf\otimes \K[G']} \simeq (R^{\Z^n}/\nf)[G'],\]
which  is a Laurent polynomial algebra over a field. This proves that $R\nf$ is a prime ideal
of $R$ for any $\nf\in \Specm(R^{\Z^n})$.
\end{proof}

\section{Multiparameter Weyl algebras and Hayashi's $q$-analog of the Weyl algebras} \label{sec:mwa}
In this section we consider a class of \emph{multiparameter Weyl
algebras} defined in \cite{Be}, which is a
 particular case of twisted multiparameter Weyl algebras. For the convenience of the reader we include 
 the definition.

\subsection{Definition}
 Assume $\un r = (r_1,\dots,r_n)$ and $\un s =
(s_1,\dots,s_n)$ are $n$-tuples of nonzero scalars in a field $\K$
such that $(r_is_i^{-1})^2 \neq 1$ for each $i$. Let $A_{\un r,\un
s}(n)$ be the unital associative algebra over the field $\K$
generated by elements $\rho_i, \rho_i^{-1}, \sigma_i,
\sigma_i^{-1}, x_i$, $y_i, \ i=1,\dots, n$, subject to the
following relations:
\begin{itemize}
\item[(R1)]
 The $\rho_i^{\pm 1}, \ \sigma_j^{\pm 1}$ all commute with one
another and  $\rho_i\rho_i^{-1}=\sigma_i\sigma_i^{-1}=1;$
\smallskip

\item[(R2)] $\rho_ix_j =r_i^{\delta_{i,j}}x_j \rho_i \qquad \rho_i
y_j = r_i^{-\delta_{i,j}}y_j \rho_i \qquad    \quad 1 \leq i,j
\leq n;$
\smallskip

\item[(R3)] $\sigma_ix_j =s_i^{\delta_{i,j}}x_j \sigma_i \qquad
\sigma_i y_j = s_i^{-\delta_{i,j}}y_j \sigma_i \qquad
   \quad 1 \leq i,j \leq n;$
\smallskip

\item[(R4)]  $x_ix_j = x_jx_i, \qquad  y_iy_j = y_jy_i,  \qquad 1
\leq i,j \leq n; $

\noindent   $y_ix_j =  x_j y_i, \qquad  1 \leq i \neq j \leq n$;
\smallskip

\item[(R5)] $y_ix_i - r_i^2 x_iy_i = \sigma_i^2$ \  and \ $y_ix_i
- s_i^2 x_iy_i = \rho_i^2$,   \qquad $1 \leq i \leq n$,

\noindent or equivalently
\smallskip

\item[(R5')] $\displaystyle{y_ix_i =  \frac{ r_i^2\rho_i^2 -s_i^2
\sigma_i^2}{r_i^2 -s_i^2}}$ \  and \ $\displaystyle{x_iy_i =
\frac{\rho_i^2 -\sigma_i^2}{r_i^2 -s_i^2} \qquad 1 \leq i \leq
n.}$
\end{itemize} \smallskip

When $r_i = q^{-1}$ and $s_i = q$ for all $i$, we may quotient by
the ideal generated by the elements  $\sigma_i \rho_i -1$,
$i=1,\dots,n$, to obtain Hayashi's $q$-analogs of the Weyl
algebras  $A_q^{-}(n)$ (see \cite{H90}).
\smallskip






\subsection{Realization as multiparameter twisted Weyl algebras}

Take $k=2$, and for all $i,j$ put $\la_{ij}=1$,
$r_{ij}=r_i^{\delta_{ij}}$, $s_{ij}=s_i^{\delta_{ij}}$,
where $r_i, s_i\in\K\backslash\{0\}$, $i=1,\ldots,n$.
Then $\MTWA{n}{k}{r}{s}{\La}$ is isomorphic to $A_{\underline{r},\underline{s}}(n)$.

Let us investigate the ring of invariants $R^{\Z^n}$. Consider a monomial
 $$u^d:=u_1^{d_1}\cdots u_n^{d_n} v_1^{d_{n+1}}\cdots v_n^{d_{2n}},$$ where $d\in\Z^{2n}$.
We have
\begin{equation}\label{eq:example1}
\si_i(u^d) = r_i^{d_i} s_i^{d_{n+i}} u^d.
\end{equation}

\subsection{Generic case}\label{sec-generic case}
Assuming that for each $i=1,\ldots,n$, the only pair $(d,d')\in\Z^2$ such that $r_i^d s_i^{d'}=1$
is the pair $(0,0)$ we
obtain that $R^{\Z^n}=\K$ and thus, by Theorem \ref{thm:family_of_simples}, 
$A_{\underline{r},\underline{s}}(n)$ is a simple ring.

\subsection{Hayashi's $q$-analogs of the Weyl algebras $A_q^{-}(n)$}
Assume instead that for all $i$, $r_i=q^{-1}$ and $s_i=q$, where $q\in\K$ is nonzero
and not a root of $1$. Then by \eqref{eq:example1},
$u^d$ is fixed by all $\si_i$ iff $d_i=d_{n+i}$ for all $i$. Thus $R^{\Z^n}=\K[w_1,\ldots,w_n]$
where $w_i:=u_iv_i$. Pick the maximal ideal $\nf:=(w_1-1,\ldots,w_n-1)$ of the
invariant subring. Then, by Theorem \ref{thm:family_of_simples}, we obtain that the quotient of $A_{\underline{r},\underline{s}}(n)$ by the two-sided ideal generated by $w_1-1, \ldots, w_n-1$
is a twisted generalized Weyl algebra which is simple. It is easy to check that this simple algebra is isomorphic to Hayashi's $q$-analogs of the Weyl algebras $A_q^{-}(n)$, see \cite{H90}.

\subsection{Connections with generalized Weyl algebras}
\label{sec:GWA}
Assume now that we are in the generic case as in subsection~\ref{sec-generic case}.
As it was observed in \cite{Be}, 
the multiparameter Weyl algebra $A_{\un r,\un s}(n)$ can be
realized as a degree $n$ generalized Weyl algebra.   For this
construction, let $D_i$ be the subalgebra of $A_{\un r,\un s}(n)$
generated by the elements $\rho_i,\rho_i^{-1},\sigma_i,
\sigma_i^{-1}$.  Thus, $D_i$ is isomorphic to $\K[\rho_i^{\pm
1},\sigma_i^{\pm 1}]$.   Set $D = D_1 \otimes D_2 \otimes \cdots
\otimes D_n$. Let  $\phi_i$ be the automorphism of $D_i$ given by

\begin{equation}\label{eq:phiidef} \phi_i(\rho_j) = r_i^{-\delta_{i,j}}\rho_j \qquad \qquad \phi_i(\sigma_i) = s_i^{-\delta_{i,j}}\sigma_i. \end{equation}

Now set
\begin{equation}\label{eq:tidef}t_i = \frac{r_i^2\rho_i^2 -s_i^2\sigma_i^2}{r_i^2 -s_i^2}, \qquad
X_i = x_i, \qquad Y_i = y_i, \end{equation} and observe that

$$Y_iX_i = t_i,  \qquad \text{and} \qquad  X_iY_i = \frac{ \rho_i^2 - \sigma_i^2}{r_i^2 -s_i^2} =
\phi_i(t_i)$$ \noindent  are just the relations in (R5)'.  The
relations in (R1) and (R4) are apparent. The identities in (R2)
and (R3) are equivalent to the statements  $Y_j d =
\phi_j^{-1}(d)Y_j, \ \ X_jd = \phi_j(d)X_j$ with $d = \rho_i$ and
$\sigma_i$.      Therefore, there is a surjection $W_n: = D(\un
\phi,\un t) \rightarrow A_{\un r,\un s}(n)$. But since $A_{\un
r,\un s}(n)$ has a presentation by (R1)-(R5), there is a
surjection $A_{\un r,\un s}(n) \rightarrow W_n$.   Since that map
is the inverse of the other one, these algebras are isomorphic.
Bavula \cite[Prop. 7]{B1}  has shown  that  a generalized Weyl
algebra  $D(\un \phi,\un t)$  is left and right Noetherian if $D$
is Noetherian, and it is a domain if $D$ is a domain. Since $D$ is
commutative and finitely generated, it is Noetherian,
 hence so are $W_n$ and $A_{\un r,\un
s}(n)$. 
Since $D$ is a domain as it can
be identified with the Laurent polynomial algebra $\K[\rho_i^{\pm
1}, \sigma_i^{\pm 1} \mid i =  1,\dots, n]$;  hence $A_{\un r,\un
s}(n)$ is a domain also.      In summary, we have

\begin{prp}\label{prop:GWA}\cite{Be}
 When the parameters $r_i,s_i$ are generic as in Section \ref{sec-generic case},  the  multiparameter Weyl algebra $A_{\un r,\un s}(n)$ is isomorphic
to the degree $n$ generalized Weyl algebra $W_n = D(\un \phi,\un
t)$ where $D$ is the $\K$-algebra generated by the elements
$\rho_i,\rho_i^{-1},\sigma_i, \sigma_i^{-1}$, $i=1,\dots,n$,
subject to the relations in (R1),  $\phi_i$ is as in
\eqref{eq:phiidef}; and the elements $t_i$ are as in
\eqref{eq:tidef}. Thus, $A_{\un r, \un s}(n)$ is Noetherian
domain.     
\end{prp}

\section{Jordan's simple localization of the multiparameter quantized Weyl algebra} \label{sec:jordan}
\subsection{Quantized Weyl algebras}
Let $\bar q=(q_1,\ldots,q_n)$ be an $n$-tuple of elements
of $\K\backslash\{0\}$.
Let $\Lambda=(\la_{ij})_{i,j=1}^n$ be an $n\times n$ matrix with
$\la_{ij}\in\K\backslash\{0\}$, multiplicatively skewsymmetric:
 $\la_{ij}\la_{ji}=1$ for all $i,j$.
The \emph{multiparameter quantized Weyl algebra of degree $n$ over $\K$}, denoted $A_n^{\bar q, \Lambda}(\K)$, is defined as the unital $\K$-algebra generated by
$x_i,y_i$, $1\le i\le n$ subject to the following defining relations:
\begin{alignat}{2}
y_iy_j&=\la_{ij}y_jy_i,    &&\qquad \forall i,j,\\
x_ix_j&=q_i\la_{ij}x_jx_i, &&\qquad i<j, \\
x_iy_j&=\la_{ji}y_jx_i,    &&\qquad i<j,\\
x_iy_j&=q_j\la_{ji}y_jx_i, &&\qquad i>j,\\
x_iy_i-q_iy_ix_i & = 1+\sum_{k=1}^{i-1}(q_k-1)y_kx_k, &&\qquad \forall i.
\end{alignat}
This algebra first appeared in \cite{M}, and was further studied
in \cite{AD} and \cite{J} among others.
For $\K=\C$ and $q_1=\cdots =q_n=\mu^2$,
$\la_{ji}=\mu\;\forall j<i$, where $\mu\in\K\backslash\{0\}$,
the algebra $A_n^{\bar q,\Lambda}(\K)$ is isomorphic to
the quantized Weyl algebra introduced by Pusz and Woronowicz \cite{PW}.

The quantized Weyl algebra can be realized as a twisted generalized Weyl algebra (first observed in \cite{MT02}) in the following way.
Let $P=\K[s_1,\ldots, s_n]$ be the polynomial algebra in
$n$ variables and $\tau_i$ the $\K$-algebra
automorphisms of $P$ defined by
\begin{equation}\label{eq:qweylsigmadef}
\tau_i(s_j)=
\begin{cases}
s_j, & j<i, \\
1+q_is_i+\sum_{k=1}^{i-1}(q_k-1)s_k, & j=i, \\
q_is_j, & j>i.
\end{cases}
\end{equation}
One can check that the $\tau_i$ commute.
Let $\mu=(\mu_{ij})_{i,j=1}^n$
be defined by 
\begin{equation}\label{eq:myn1}
\mu_{ij}= \begin{cases}
\la_{ji}, & i<j, \\
q_j\la_{ji},& i>j.
\end{cases}
\end{equation}
Put $\tau=(\tau_1,\ldots,\tau_n)$ and
$s=(s_1,\ldots,s_n)$. Let $\TGWA{P}{\tau}{s}{\mu}$ be the associated twisted generalized Weyl algebra.
From \eqref{eq:qweylsigmadef} it is easy to see that $\TGWA{P}{\tau}{s}{\mu}$ is $\K$-finitistic, and that the minimal polynomials are $p_{ij}(x)=x-1$ for $i<j$ and $p_{ij}(x)=x-q_i$ for $j>i$, so the algebra is of type $(A_1)^n$.
By Theorem \ref{thm:A1n_pres}, one checks that $\TGWA{P}{\tau}{s}{\mu}$ is isomorphic to $A_n^{\bar q, \Lambda}(\K)$ via $X_i\mapsto x_i$, $Y_i\mapsto y_i$ and $s_i\mapsto y_ix_i$.
The representation theory of $A_n^{\bar q, \Lambda}$ has been studied from the point of view of TGW algebras in \cite{MT02} and \cite{Ha06}.

In the following it will be convenient to identify $P$ with its isomorphic image in $A_n^{\bar q,\Lambda}$ via $s_i\mapsto y_ix_i$. 
Consider the following elements in $A_n^{\bar q,\Lambda}$:
\begin{equation}\label{eq:zidef}
z_i=1+\sum_{k\le i}(q_k-1)s_k,\qquad i=1,\ldots,n.
\end{equation}
It was shown in \cite{J} that the the set $Z:=\{z_1^{k_1}\cdots z_n^{k_n}\mid k_1,\ldots,k_n\in\Z\}$ is an Ore set in $A_n^{\bar q,\Lambda}$ and that,
provided that none of the $q_i$ is a root of unity,
the localized algebra 
\[B_n^{\bar q,\Lambda}:=Z^{-1}A_n^{\bar q,\Lambda}\]
is simple.

The algebra $B_n^{\bar q,\Lambda}$ can also be realized as a twisted generalized Weyl algebra. To see this, consider the following subset of $P$:
\begin{equation}\label{eq:locS}
S=\{\alpha z_1^{k_1}\cdots z_n^{k_n}\mid \alpha\in\K\backslash\{0\},
k_i\in\Z\}.
\end{equation}
where $z_i$ were defined in \eqref{eq:zidef}.
Then $0\notin S$, $1\in S$, $a,b\in S\Rightarrow ab\in S$, the elements
of $S$ are regular, and moreover $S$ has the virtue of being $\Z^n$-invariant, using the relation
\begin{equation}\label{eq:taizj}
\ta_i(z_j)=\begin{cases}
z_j,& j<i,\\
q_iz_j,& j\ge i,
\end{cases}\end{equation}
which can be proved using \eqref{eq:zidef} and \eqref{eq:qweylsigmadef}.
Thus
\cite[Thm.~5.2]{FutHar2011} 
 can be applied to give, together with the isomorphism
 $A_n^{\bar q,\Lambda}\simeq\TGWA{P}{\tau}{s}{\mu}$, that
\[S^{-1}A_n^{\bar q,\Lambda}\simeq S^{-1}\TGWA{P}{\tau}{s}{\mu}\simeq \TGWA{S^{-1}P}{\tilde\tau}{s}{\mu}.\]
But localizing at $S$ is equivalent to localizing at $Z$, and thus
\[B_n^{\bar q,\Lambda}\simeq \TGWA{S^{-1}P}{\tilde\tau}{s}{\mu}.\]

\subsection{Relation to multiparameter twisted Weyl algebras}
We show here how the algebra $B_n^{\bar q,\Lambda}$ fits into the framework of multiparameter twisted Weyl algebras. We keep all notation from previous section.
Let $\bar q=(q_1,\ldots,q_n)\in(\K\backslash\{0\})^n$ and let $\Lambda=(\la_{ij})_{i,j=1}^n$ be an $n\times n$ matrix with
$\la_{ij}\in\K\backslash\{0\}$, $\la_{ii}=1$, $\la_{ij}\la_{ji}=1$ for all $i,j$.
We assume that none of the $q_i$ is a root of unity.
Let $k=1$ and put
\begin{equation}\label{eq:qweyl_rs_def}
r_{ij}=\begin{cases}1,& j\le i\\ q_i^{-1},& j>i\end{cases}
\qquad\qquad
s_{ij}=\begin{cases}1,& j<i\\ q_i^{-1},&j\ge i\end{cases}
\end{equation}
Then conditions \eqref{eq:mtwa_conds} are satisfied.
Let $\MTWA{n}{k}{r}{s}{\La}$ be the corresponding multiparameter twisted Weyl algebra. Recall that, by definition, this means that 
$\MTWA{n}{k}{r}{s}{\La}$ is the twisted generalized Weyl algebra
$\TGWA{R}{\si}{t}{\mu}$ where
\begin{gather}
R=\K[u_1^{\pm 1},\ldots,u_n^{\pm 1}, v_1^{\pm 1},\ldots,v_n^{\pm 1}],\\
\si_i(u_j)=r_{ij}^{-1}u_j,\qquad \si_i(v_j)=s_{ij}^{-1}v_j,\\
t_i=\frac{u_i - q_i^{-1}v_i}{1-q_i^{-1}},\\
\mu_{ij}=r_{ji}^{-k}\la_{ji}, \label{eq:myn2}
\end{gather}
for all $i,j\in\{1,\ldots,n\}$.
Note that the $\mu_{ij}$ in \eqref{eq:myn2} coincides with
the ones defined in \eqref{eq:myn1}.
The goal now is to explain the following diagram, which proves Theorem C stated in the introduction.
\[
\xymatrix@C=0.5cm@M=1ex{
 (R,\si,t) \ar@{->>}[d]_{\pi} \ar@{->>}[drr]^{\psi}   &&              \\
 (R/J,\bar\si,\bar t)  \ar@<.5ex>@{->>}[rr]^{\Psi} &&
 (S^{-1}P,\tilde\tau,s)  \ar@<.5ex>@{->}[ll]^{\Phi}   \\
                          &\ar@{|->}[d]^*+{\Af{\mu}}& (P,\tau,s) \ar[ull]^{\varphi} \ar@{^{(}->}[u]_{\iota} \\
\MTWA{n}{k}{r}{s}{\Lambda} =  \TGWA{R}{\si}{t}{\mu} \ar@{->>}[d] \ar@{->>}[drr]
   &&                       \\
\frac{\MTWA{n}{k}{r}{s}{\Lambda}}{\langle J\rangle}\simeq \TGWA{R/J}{\bar\si}{\bar t}{\mu}  \ar@<.5ex>@{->}[rr]^{\simeq}
   && \TGWA{S^{-1}P}{\tilde\tau}{s}{\mu}\simeq B_n^{\bar q,\Lambda} \ar@<.5ex>@{->}[ll]   \\
                          && \TGWA{P}{\tau}{s}{\mu}\simeq A_n^{\bar q,\Lambda} \ar@{^{(}->}[ull] \ar@{^{(}->}[u] 
}
\]
\subsubsection{The map $\psi$}
Define a $\K$-algebra morphism
\[\psi:R\to S^{-1}P,\quad
 \psi(u_i)=-q_i^{-1}z_{i-1}, \quad \psi(v_i)=-z_i,  \quad i=1,\ldots,n\]
where $z_0:=1$.
We claim that $\psi$ is $\Z^n$-equivariant. Indeed,
\[\psi(\si_i(u_j))= \psi(r_{ij}^{-1}u_j)=-r_{ij}^{-1}q_i^{-1}z_{i-1},\]
while, using \eqref{eq:taizj} and \eqref{eq:qweyl_rs_def} in the last step,
\[\tilde\tau_i(\psi(u_j))=\tilde\tau_i (-q_j^{-1}z_{j-1})=-r_{ij}^{-1}q_i^{-1}z_{i-1}.\]
Similarly $\psi(\si_i(v_j))=\tilde\tau_i(\psi(v_j))$. This proves that $\psi\si_i=\tilde\tau_i\psi$ for each $i$, so in other words, that $\psi$ is $\Z^n$-equivariant. Also, for any $i\in\{1,\ldots,n\}$,
\[\psi(t_i)=\psi\big
(\frac{u_i - q_i^{-1}v_i}{1-q_i^{-1}}\big)=
\frac{-q_i^{-1}z_{i-1}+q_i^{-1}z_i}{1-q_i^{-1}}=\frac{z_i-z_{i-1}}{q_i-1}=
s_i
\]
by \eqref{eq:zidef}. We have proved that $\psi$ is a morphism in the category $\TGW{n}{\K}$ between $(R,\si,t)$ and $(S^{-1}P,\tilde\tau,s)$.
It is easy to see that $\psi$ is surjective because the image contains both $s_1,\ldots,s_n$ since $\psi(t_i)=s_i$, and the inverses of the $z_i$: $z_i^{-1}=\psi(-v_i^{-1})$.
Applying the functor $\mathcal{A}$ to $\psi$ gives a surjective $\K$-algebra
morphism $\Af{\mu}(\psi):\TGWA{R}{\si}{t}{\mu}\to\TGWA{S^{-1}P}{\tau}{s}{\mu}$.

\subsubsection{The map $\pi$}
We determine the invariant subring $R^{\Z^n}$. For any $i\in\{1,\ldots,n\}$ and $d\in\Z^{2n}$ we have
\[\si_i(u^d)= q_i^{-\sum_{j>i}d_{j} - \sum_{j\ge i}d_{n+j}}u^d\]
Thus $u^d\in R^{\Z^n}$ iff for each $i=1,\ldots,n$ we have
$d_{n+i} + \sum_{j=i+1}^n(d_{j}+d_{n+j})=0$. This system of equations is equivalent
to that $d_{2n}=0,\; d_{2n-1}+d_{n}=0,\; d_{2n-2}+d_{n-1}=0, \ldots,\; d_{n+1}+d_{2}=0$.
Thus $R^{\Z^n}=\K[w_1,\ldots,w_n]$ where $w_1=-u_1,\; w_2=u_2v_1^{-1}, \ldots,\; w_n=u_nv_{n-1}^{-1}$.
Pick \[\nf:=(w_1-q_1^{-1},\ldots, w_n-q_n^{-1})\in \Specm(R^{\Z^n}).\]
Let $J=R\nf$ be the ideal in $R$ generated by $\nf$.
The canonical map $\pi:R\to R/J$ is $\Z^n$-equivariant and maps $t_i$ to $\bar t_i=t_i+J$. 

\subsubsection{The map $\Psi$}
We have $\psi(w_i)=q_i^{-1}$ for $i=1,\ldots,n$ which shows that $J=R\nf\subseteq\ker\psi$. Thus $\psi$ induces a map $\Psi:R/J\to S^{-1}P$,
also $\Z^n$-equivariant and $\Psi(\bar t_i)=s_i$. Since $\psi$ is surjective, so is $\Psi$. Applying the functor from
\cite[Thm.~3.1]{FutHar2011} 
 we get a surjective homomorphism $\mathcal{A}(\Psi):\TGWA{R/J}{\si}{t}{\mu}\to \TGWA{S^{-1}P}{\tilde\tau}{s}{\mu}$. However, by Theorem \ref{thm:family_of_simples}, the algebra $\TGWA{R/J}{\si}{t}{\mu}$ is simple, and thus $\Psi$ is an isomorphism.
 
\subsubsection{The maps $\varphi, \iota, \Phi$}
Similarly one can show that the map $\varphi:P\to R/J$ defined
by $\varphi(s_i)=\bar t_j$ is $\Z^n$-equivariant and that the elements
of $S$ are mapped to invertible elements of $R/J$, showing that
$\varphi$ factorizes through the canonical map $\iota:P\to S^{-1}P$,
 inducing a map $\Phi$. Applying the functor $\Af{\mu}$ gives
 corresponding homomorphisms of twisted generalized Weyl algebras.

\section{Simple weight modules}\label{sec:weight}

In this section we describe the simple weight modules over
the simple algebras $A=\MTWA{n}{k}{r}{s}{\La}/\langle J\rangle$
from Theorem \ref{thm:family_of_simples}.
We also assume that the ground field $\K$ is algebraically closed.
We will use notation from Section \ref{sec:MTWA_def}.

\subsection{Dynamics of orbits and their breaks}
The group $\Z^n$ acts on $R$ via the automorphisms $\si_i$. Explicitly,
$g(r)=(\si_1^{g_1}\cdots\si_n^{g_n})(r)$ for $g=(g_1,\ldots,g_n)
\in\Z^n$ and $r\in R$. Using this action and \eqref{eq:tgwarels1}
we have $a\cdot r=(\deg a)(r)\cdot a$ for any homogenous $a\in A$
and any $r\in R$.
The group $\Z^n$ also acts on $\Max(R)$, the set of maximal ideals of $R$.
Let $\Om$ denote the set of orbits of this action.
An element $\mf\in\Max(R)$ is called an \emph{$i$-break} if $t_i\in \mf$.
An orbit $\mathcal{O}\in\Om$ is called \emph{degenerate} if it contains
an $i$-break for some $i$.
A break $\mf$ in an orbit $\mathcal{O}$ is called \emph{maximal}
if $\mf$ is an $i$-break for all $i$ for which $\mathcal{O}$ contains an $i$-break.

\begin{prp}\label{prp:stabilizer}
Let $\mf\in\Specm(R)$. Then the stabilizer $\Stab(\mf)$
is trivial.
\end{prp}
\begin{proof}
Write 
\[\mf=(\bar u_i-\al_i,\bar v_i-\be_i\mid i=1,\ldots,n),\]
where $\bar u_i=u_i+J, \bar v_i=v_i+J$ and $\al_i,\be_i\in\K\backslash\{0\}$.
Suppose $g\in\Stab(\mf)$. Then
\begin{gather}\label{eq:sigmf_calc}
\begin{aligned}
\si_g(\mf)&=(\si_g(\bar u_i)-\al_i, \si_g(\bar v_i)-\be_i \mid i=1,\ldots,n)=\\
&=\big((r_{1i}^{g_1}\cdots r_{ni}^{g_n})^{-1}\bar u_i -\al_i,
(s_{1i}^{g_1}\cdots s_{ni}^{g_n})^{-1}\bar v_i -\be_i \mid i=1,\ldots,n \big)
 \end{aligned}
 \end{gather}
Thus 
\[r_{1i}^{g_1}\cdots r_{ni}^{g_n}=
s_{1i}^{g_1}\cdots s_{ni}^{g_n}=1
\]
Raising all sides to the $k$:th power and using that
$r_{ij}^k=s_{ij}^k$ for all $i\neq j$, we obtain
that $r_{ii}^{kg_i}=s_{ii}^{kg_i}=1$ for all $i$
which, since $r_{ii}/s_{ii}$ is not a root of unity,
implies that $g_i=0$ for all $i$.
\end{proof}

\begin{prp}\label{prp:breakstructure}
Consider the maximal ideal
\[\mf=(u_1-\al_1,\ldots,u_n-\al_n,v_1-\be_1,\ldots,v_n-\be_n)\in\Specm(R).\]
Then, for all $i\in\{1,\ldots,n\}$ and all $g=(g_1,\ldots,g_n)\in\Z^n$,
\begin{equation}
t_i\in\si_g(\mf)\Longleftrightarrow
(\al_i/\be_i)^k=(s_{ii}/r_{ii})^{(g_i+1)k}.
\end{equation}
\end{prp}
\begin{proof}
By the calculation \eqref{eq:sigmf_calc} and the definition
 \eqref{eq:mtwa_ti_def} of $t_i$
we have $t_i\in\si_g(\mf)$ iff 
\[\Big(\frac{r_{1i}^{g_1}\ldots r_{ni}^{g_n}\al_i}{s_{1i}^{g_1}\ldots s_{ni}^{g_n}\be_i}\Big)^k=(s_{ii}/r_{ii})^k\]
Using that $r_{ij}^k=s_{ij}^k$ for $i\neq j$ and simplifying, the
claim follows.
\end{proof}
\begin{cor}\label{cor:breakstructure}
If $t_i\in\mf$, then for $g=(g_1,\ldots,g_n)\in\Z^n$,
\[t_i\in\si_g(\mf)\Longleftrightarrow g_i=0.\]
\end{cor}

\begin{cor} \label{cor:maxexistence}
Every degenerate orbit contains a maximal break.
\end{cor}

\begin{rem} Corollary \ref{cor:maxexistence} holds for any TGW algebra
 of Lie type $(A_1)^n$ using the fact that $\si_j(t_i)=\ga_{ji} t_i$ for any $j\neq i$.
\end{rem}

\subsection{General results on simple weight modules with no proper inner breaks}
We collect here some notation and results from \cite{Ha06}.

Let $A=\TGWA{R}{\si}{t}{\mu}$ be a twisted generalied Weyl algebra.
Let $V$ be a simple weight module over $A$.

\begin{dfn}[\cite{Ha06}]
$V$ has \emph{no proper inner breaks} if for
any $\mf\in\Supp(V)$ and any homogenous $a$ with $aM_\mf\neq 0$
we have $a' a\notin \mf$ for some homogenous $a'$ with $\deg(a')=-\deg(a)$.
\end{dfn}
This definition is slightly different than the one given in \cite{Ha06} but
can be proved to be equivalent.
Consider the following sets (also equivalent to the definitions in \cite{Ha06}), defined for any $\mf\in\Specm(R)$.
\begin{align}
\tilde G_\mf &:= \{g\in \Z^n\mid \text{$A_{-g}A_g$ is not contained in $\mf$}\},\\
G_\mf &:= \tilde G_\mf \cap \Stab_{\Z^n}(\mf).
\end{align}
One can show that $G_\mf$ is a subgroup of $\Z^n$ and $\tilde G_\mf$ is a union of cosets from $\Z^n/G_\mf$.

Fix now $\mf\in\Supp(V)$.
One checks that the subalgebra $B(\mf):=\bigoplus_{g\in\Stab(\mf)} A_g$
of $A$ preserves the weight space $V_\mf$.
For any $g\in \tilde G_\mf$ we pick elements $a_g\in A_g$ and $a_g'\in A_{-g}$ such that $a_g'a_g\notin \mf$.
The following theorem describes the simple weight modules with no proper inner breaks up to the structure of $V_\mf$ as a $B(\mf)$-module.
\begin{thm}[\cite{Ha06}] \label{thm:NPIB}
Suppose $V$ has no proper inner breaks. If $\{v_i\}_{i\in J}$ is a $\K$-basis for $V_\mf$ ($J$ some index set), then the following is a $\K$-basis for $V$:
\begin{equation}
C:=\{a_g v_i\mid g\in S, i\in J\}
\end{equation}
where $S\subseteq\tilde G_\mf$ is a set of representatives for $\tilde G_\mf$ modulo $G_\mf$.
Moreover,
for any $v\in V_\mf$, any $i\in\{1,\ldots,n\}$ and $g\in S$ we have
\begin{equation}
X_ia_g v = \begin{cases} a_h b_{g,i} v,& g+e_i\in\tilde G_{\mf}\\
0,&\text{otherwise}, \end{cases}
\qquad
Y_ia_g v = \begin{cases} a_k c_{g,i} v,& g-e_i\in\tilde G_{\mf}\\
0,&\text{otherwise}. \end{cases}
\end{equation}
where $h,k\in S$ with $h\in (g+e_i)+ G_\mf$ and $k\in (g-e_i)+G_\mf$
and $b_{g,i},c_{g,i}\in B(\mf)$ are given by
\begin{equation}\label{eq:bgicgi}
b_{g,i}=\si_{-h}(X_ia_ga_{g+e_i-h}'a_h')a_{g+e_i-h},
\qquad
c_{g,i}=\si_{-k}(Y_ia_ga_{g-e_i-k}'a_k')a_{g-e_i-k}.
\end{equation}
\end{thm}

\subsection{The case of trivial stabilizer}
We show here a theorem which implies that all simple weight modules
over $\MTWA{n}{k}{r}{s}{\La}/\langle\nf\rangle$ have no proper inner breaks.
\begin{thm} If $V$ is a simple weight module over a twisted generalized Weyl algebra $\TGWA{R}{\si}{t}{\mu}$ such that the stabilizer $\Stab(\mf)$ is trivial for some (hence all) weight $\mf\in\Supp(V)$, then $V$ has no proper inner breaks.
\end{thm}
\begin{proof} Suppose $\mf\in\Supp(V)$ has trivial stabilizer. Let $g\in\Z^n$ and assume
$a\in A_g$ is such that $aV_\mf\neq 0$. Since $V$ is simple, $V_\mf\cap AaV_\mf\neq 0$.
But $V_\mf\cap AaV_\mf\subseteq A_{-g}aV_\mf$ since $\mf$ has trivial stabilizer.
This shows that there exists an element $b\in A_{-g}$ such that $ba V_{\mf}\neq 0$.
Since $\deg(ba)=0$ we have $ba\in R$. Then $baV_\mf\neq 0$ implies $ba\notin\mf$.
\end{proof}

\subsection{Abstract description of the simple weight modules in case of trivial stabilizer}
Let $A=\TGWA{R}{\si}{t}{\mu}$ be a TGWA
where $\mu$ is symmetric.
In \cite{MPT} a description of all simple weight modules with support in an orbit with trivial
stabilizer is given in terms of a Shapovalov type form. The form used in \cite{MPT} requires
the matrix $\mu$ to be symmetric (due to its formulation in terms of a certain involution on the TGWA).
As is observed in \cite{HO}, there is another way to define a bilinear form which works for general $\mu$.
It is given as follows. Let $\mathfrak{p}_0:A\to A_0=R$ be the graded projection onto the degree zero
component of $A$ with respect to the standard $\Z^n$-gradation on $A$. Then put
\begin{equation}
 F:A\times A\to R,\qquad F(a,b)=\mathfrak{p}_0(ab).
\end{equation} 
Such forms have been studied for arbitrary group graded rings \cite{CR}.

We have the following result.
\begin{thm}
Let $A=\TGWA{R}{\si}{t}{\mu}$ be any twisted generalized Weyl algebra.
Let $V$ be any simple weight module over $A$ such that $\Stab(\mf)=\{0\}$
for $\mf\in\Supp(V)$. Then $V\simeq A/N(\mf)$ where $A$ is considered as
a left module over itself and $N(\mf)$ is the left ideal given by
\begin{equation}
N(\mf)=\{a\in A\mid F(b,a)\in\mf \;\forall b\in A\}
\end{equation}
\end{thm}
\begin{proof}
Similar to the case of symmetric $\mu$ proved in
\cite[Lemma 6.1 and Corollary 6.2]{MPT}.
\end{proof}

\subsection{Bases and explicit action on the simple weight modules over $\MTWA{n}{k}{r}{s}{\La}$}
Let $n,k,r,s,\La$ be as in Section \ref{sec:MTWA_def}. Assume that for each $i=1,\ldots,n$, the scalar $r_{ii}/s_{ii}$ is not a root of unity.
Let $R,\si,t,\mu$ be as in Section \ref{sec:MTWA_def}.

Let $J$ be any $\Z^n$-invariant ideal of $R$.
Let $A=\TGWA{R/J}{\bar \si}{\bar t}{\mu}$.
Thus for $J=0$, $A$ equals the multiparameter twisted Weyl algebra $\MTWA{n}{k}{r}{s}{\La}$, and for $J=R\nf$ where $\nf\in\Specm(R^{\Z^n})$, $A$ equals a simple quotient of the algebra in the former case.

We will describe the simple weight modules over $A$, using Theorem \ref{thm:NPIB}

Let $V$ be a simple weight module over $A$.
Let $\mf\in\Supp(V)$. Since $\K$ is algebraically closed we have
\[\mf=(\bar u_i-\al_i, \bar v_i-\be_i\mid i=1,\ldots,n)\]
where $\bar u_i=u_i+J, \bar v_i=v_i+J$ and $\al_i,\be_i\in\K\backslash\{0\}$
for $i=1,\ldots,n$.

We determine the set $\tilde G_\mf$. Let $g\in\Z^n$.
Since $A_g=\bar R Z^{(g)}$ (where $Z^{(g)}=Z_1^{(g)}\cdots Z_n^{(g_n)}$ 
where $Z_i^{(j)}$ equals $X_i^j$ if $j\ge 0$ and $Y_i^{-j}$ otherwise) 
and $\forall i\neq j: \;\si_i(t_j)=\ga_{ij}t_i$ for some $\ga_{ij}\in\K\backslash\{0\}$  it is clear that
\begin{align*}
\tilde G_\mf &= \{g\in\Z^n\mid Z^{(-g)}Z^{(g)}\notin \mf\}=\\
&=\{g\in\Z^n\mid Z_1^{(-g_1)}Z_1^{(g_1)}\cdots Z_n^{(-g_n)}Z_n^{(g_n)}\notin\mf\}=\\
&=\{g\in\Z^n\mid Z_i^{(-g_i)}Z_i^{(g_i)}\notin\mf \;\forall i\}=\\
&=\tilde G_\mf^{(1)}\times\cdots\times \tilde G_\mf^{(n)},
\end{align*}
where
\begin{equation}\label{eq:tildeGmi}
\tilde G_\mf^{(i)}:= \{g\in\Z^n\mid Z_i^{(-g_i)}Z_i^{(g_i)}\notin\mf\}.
\end{equation} 
For $j>0$ we have
\begin{equation}\label{eq:litenkalkyl1}
Z_i^{(-j)}Z_i^{(j)}=Y_i^jX_i^j=t_i\si_i^{-1}(t_i)\cdots\si_i^{-j+1}(t_i)
\end{equation}
while for $j<0$,
\begin{equation}\label{eq:litenkalkyl2}
Z_i^{(-j)}Z_i^{(j)}=X_i^{-j}Y_i^{-j}=\si_i(t_i)\si_i^2(t_i)\cdots
\si_i^{-j}(t_i).
\end{equation}
So, since $\mf$ is maximal, hence prime, we see that if $j>0$ and $j\in \tilde G_\mf^{(i)}$ then $\{0,1,\ldots,j\}\subseteq \tilde G_\mf^{(i)}$.
Similarly if $j<0$ and $j\in\tilde G_\mf^{(i)}$ then $\{j,j+1,\ldots,0\}\subseteq \tilde G_\mf^{(i)}$.

We distinguish between three possibilities.
The first case is that $\tilde G_\mf ^{(i)}=\Z$. Then we say that
(the support of) $V$ is \emph{generic} in the $i$:th direction.
The second case is $j\notin\tilde G_\mf^{(i)}$ for some positive integer $j$.
Assuming $j$ is the smallest such integer,
 by \eqref{eq:tildeGmi} and \eqref{eq:litenkalkyl1} we get $\si_i^{-j+1}(t_i)\in \mf$.
By Corollary \ref{cor:breakstructure} it follows that $\si_i^{m}(t_i)\notin\mf$ for
all integers $m\neq j$. Thus $\tilde G_\mf^{(i)}=\{m\in\Z\mid m\le j-1\}$.
By Theorem \ref{thm:NPIB}, $\Supp(V)=\{\si_g(\mf)\mid g\in G_\mf\}$
and thus we can replace $\mf$ by $\si_i^{k-1}(\mf)$. Doing this, the new $j$ just equals $1$
and $G_\mf^{(i)}=\Z_{\le 0}$.
We say that $\mf$ is a \emph{highest weight} for $V$ in the $i$:th direction.
The final case is that $j\notin\tilde G_\mf^{(i)}$ for some negative integer $j$.
This is analogous to the previous case and leads to that, without loss of generality,
$\tilde G_\mf^{(i)}=\Z_{\ge 0}$ in which case we say that $\mf$ is a \emph{lowest weight}
for $V$ in the $i$:th direction.

In other words, there is an $\mf\in\Supp(V)$ such that
the shape of the support of $V$ is characterized by a vector
\begin{equation}
\tau\in\{-1,0,1\}^n
\end{equation}
via the relation
\begin{equation}
\tilde G_\mf^{(i)}=\{j\in\Z\mid j\cdot \tau_i \ge 0\}\qquad\forall i\in\{1,\ldots,n\}.
\end{equation}

 Since the stabilizer of $\mf$ is trivial by Proposition
\ref{prp:stabilizer}, the subalgebra $B(\mf)$ in Theorem \ref{thm:NPIB} is just $R$.
From well known results \cite{DFO} (see \cite[Proposition 7.2]{MPT} for a proof in the TGW
algebra case),
it follows that $V_\mf$ is simple as a $B(\mf)$-module since $V$ is simple as an $A$-module.
 Thus, since $R/\mf=\K$, we have $\dim_\K V_\mf=1$.
Pick $v_0\in V_\mf,\, v_0\neq 0$. Then Theorem \ref{thm:NPIB} implies that the set
\begin{equation}
C=\{v_g:=Z_1^{(g_1)}\cdots Z_n^{(g_n)} v_0\mid g=(g_1,\ldots ,g_n)\in \tilde G_\mf \}
\end{equation}
is a $\K$-basis for $V$, where $Z_i^{(j)}=X_i^j$ if $j\ge 0$ and $Y_i^{-j}$ otherwise.
Furthermore, the action of $X_i$, $Y_i$ on the elements of $C$ is given by
\begin{equation}
X_iv_g = \begin{cases} b_{g,i} v_{g+e_i},& \text{if $(g_i+1)\tau_i\ge 0$},\\
0,&\text{otherwise}, \end{cases}
\qquad
Y_iv_g = \begin{cases} c_{g,i} v_{g-e_i},& \text{if $(g_i-1)\tau_i\ge 0$},\\
0,&\text{otherwise}, \end{cases}
\end{equation}
for certain $b_{g,i}, c_{g,i}\in\K$. Although the formulas \eqref{eq:bgicgi}
can be used to calculate these scalars, one can also use a more direct approach
which is available due to our knowledge of the commutation relations \eqref{eq:mtwarels} among
the generators $X_i,Y_i$ in $A$. Straightforward calculation gives the following.
\begin{align}
b_{g,i}&=
 \ga_{i1}^{(g_1)}\cdots\ga_{i,i-1}^{(g_{i-1})} \cdot
 r_{i+1,i}^{kg_i}\cdots r_{ni}^{kg_n} \cdot 
\begin{cases}
1&\text{if $g_i\ge 0$,}\\
 \frac{r_{ii}^{(1-g_i)k}\al_i^k-s_{ii}^{(1-g_i)k}\be_i^k}{r_{ii}^k-s_{ii}^k}
&\text{if $g_i<0$,}
\end{cases} \\
c_{g,i}&=\ep_{i1}^{(g_1)}\cdots \ep_{i,i-1}^{(g_{i-1})} \cdot
 r_{i+1,i}^{kg_i}\cdots r_{ni}^{kg_n} \cdot
\begin{cases}
 \frac{r_{ii}^{kg_i}\al_i^k-s_{ii}^{kg_i}\be_i^k}{r_{ii}^k-s_{ii}^k}
  &\text{if $g_i\ge 0$,}\\
 1 & \text{if $g_i<0$,}
\end{cases} 
\end{align}
where
\begin{equation}
\ga_{ij}^{(l)}=
\begin{cases}
\big( (r_{ji}/r_{ij})^k\la_{ij}\big) ^l, &l\ge 0,\\
\big( r_{ji}^k\la_{ij}\big) ^l,& l<0,
\end{cases}
\qquad
\ep_{ij}^{(l)}=
\begin{cases}
\big(r_{ij}^k\la_{ji}\big)^l, & l \ge 0,\\
\la_{ji}^l,&l<0.
\end{cases}
\end{equation}

\section{Whittaker modules} \label{sec:whittaker}
\begin{dfn}
Let $A$ be a twisted generalized Weyl algebra of degree $n$.
A module $V$ over $A$ is called a \emph{Whittaker module}
if there exists a vector $v_0\in V$ (called \emph{Whittaker vector)}
 and nonzero scalars $\zeta_1,\ldots,\zeta_n\in\K\backslash\{0\}$
 such that the following conditions hold:
\begin{itemize}
\item $V=Av_0$,
\item $X_iv_0=\zeta_i v_0$ for each $i=1,\ldots,n$.
\end{itemize}
The pair $(V,v_0)$ is called a \emph{Whittaker pair of type $(\zeta_1,\ldots,\zeta_n)$}.
A morphism of Whittaker pairs $(V,v_0)\to (W,w_0)$ is an $A$-module morphism $V\to W$
mapping $v_0$ to $w_0$.
\end{dfn}

The following theorem describes Whittaker pairs over a family of TGWAs which
properly includes all generalized Weyl algebras in which the $t_i$ are regular.
It is a generalization of \cite[Theorem 3.12]{BO}.
We use the notation from Section \ref{sec:A1n_pres}.
\begin{thm}
Let $A=\TGWA{R}{\si}{t}{\mu}$ be a $\K$-finitistic TGW algebra of Lie type $(A_1)^n$.
Assume that $(R,\si,t)$ is $\mu$-consistent and that $R$ is Noetherian.
\begin{enumerate}[{\rm (a)}]
\item If $A$ has a Whittaker module, then
\begin{equation}\label{eq:Whittaker_cond}
\text{$\ga_{ij}=\mu_{ij}$ for all $i\neq j$.}
\end{equation}
\item Conversely, if \eqref{eq:Whittaker_cond} holds, then for each $\zeta\in(\K\backslash\{0\})^n$,
there is a bijection
\begin{align*}
\left\{\begin{gathered}
\text{Isomorphism classes $[(V,v_0)]$ of}\\
\text{Whittaker pairs of type $\zeta$}
\end{gathered}\right\}
&\overset{\Psi}{\longrightarrow}
\left\{\begin{gathered} \text{Proper $\Z^n$-invariant left ideals $Q$ of $R$} \end{gathered}\right\}\\
[(V,v_0)] &\longmapsto \Ann_R v_0 \\
[(R/Q, 1+Q)] &  \longleftarrow Q
\end{align*}
where $R/Q$ is given an $A$-module structure by 
\begin{gather}\label{eq:RQAmod}
\begin{aligned}
s.\bar r &= \overline{sr} \quad\forall s\in R,\\
X_i. \bar r &=  \zeta_i \overline{\sigma_i(r)}, \\
Y_i. \bar r &=  \zeta_i^{-1} \overline{\sigma_i^{-1}(r) t_i},
\end{aligned}
\end{gather}
for all $\bar r\in R/Q$, where $\bar r:=r+Q\in R/Q$ for $r\in R$.
\item[c)] Futhermore, there is a morphism of Whittaker pairs $(V,v_0)\to (W,w_0)$
iff $\Psi\big([(V,v_0)]\big)\subseteq\Psi\big([(W,w_0)]\big)$.
\end{enumerate}
\end{thm}
\begin{proof}
a) Suppose $(V,v_0)$ is a Whittaker pair with respect to $\zeta\in(\K\backslash\{0\})^n$.
Then for $i\neq j$, $X_iX_j v_0 = \zeta_i\zeta_j v_0$. On the other hand,
by relation \eqref{eq:pres_rels2}, $X_iX_j=\ga_{ij}\mu_{ij}^{-1} X_jX_i$ and thus
$X_iX_j v_0 = \ga_{ij}\mu_{ij}^{-1}X_jX_iv_0= \ga_{ij}\mu_{ij}^{-1} \zeta_i\zeta_j v_0$.
Thus, since $v_0$ and all $\zeta_i$ are nonzero by definition, we conclude 
that \eqref{eq:Whittaker_cond} must hold.

b) Suppose $(V,v_0)$ is a Whittaker pair with respect to $\zeta\in(\K\backslash\{0\})^n$.
Let $Q=\Ann_R v_0$. Clearly $Q$ is a proper left ideal of $R$. 
For any $r\in Q$ we have $0=X_i rv_0 = \si_i(r)X_i v_0 = \zeta_i\si_i(r) v_0$ which shows that
$\si_i(Q)\subseteq Q$ for any $i\in\{1,\ldots,n\}$.
Since $R$ is Noetherian, $\si_i^{-1}(Q)\subseteq Q$ as well,
which proves that $Q$ is $\Z^n$-invariant.
In addition, if $(V,v_0)$ and $(W,w_0)$ are two isomorphic Whittaker pairs, then
clearly $\Ann_R v_0= \Ann_R w_0$. This shows that the map $\Psi$ is well-defined.

To prove that $\Psi$ is surjective,
 suppose that $Q$ is a proper $\Z^n$-invariant left ideal of $R$.
We show that \eqref{eq:RQAmod} extends the natural $R$-module structure on
$R/Q$ to an $A$-module structure. We only prove that the following
 relations are preserved:
 $X_iY_j=\mu_{ij}Y_jX_i$ ($i\neq j$) and
 $Y_jY_i=\ga_{ij}\mu_{ji}^{-1} Y_iY_j$ ($i\neq j$).
  The other cases are identical to the generalized Weyl algebra case
considered in \cite[Section 3]{BO}. We have
\[X_iY_j .\bar r = X_i.\zeta_j^{-1}\overline{\si_j^{-1}(r)t_j} =
 \zeta_i\zeta_j^{-1} \overline{\si_i\si_j^{-1}(r)\si_i(t_j)}.\]
Using that $\si_i(t_j)=\ga_{ij}t_j$
(see \eqref{eq:A1nsigmagamma})
 and condition \eqref{eq:Whittaker_cond}
we see that $X_iY_j.\bar r=\mu_{ij}Y_jX_i.\bar r$ for any $\bar r\in R/Q$.
Similarly $Y_jY_i.\bar r = \zeta_i^{-1}\zeta_j^{-1}\overline{\si_i^{-1}\si_j^{-1}(r) \si_j^{-1}(t_i)t_j}$
so using $\si_j^{-1}(t_i)t_j=\ga_{ji}^{-1}\ga_{ij} t_i\si_i^{-1}(t_j)$
and \eqref{eq:Whittaker_cond} again, we see that
$Y_jY_i.\bar r=\ga_{ij}\mu_{ji}^{-1} Y_iY_j.\bar r$, $\forall i\neq j$.
Thus $R/Q$ becomes an $A$-module which is a Whittaker module of type $\zeta$
with Whittaker vector $1+Q$.

To prove that $\Psi$ is injective we may, as in \cite{BO},
construct a universal Whittaker module $V_{\mathfrak{u}}$ of type $\zeta$ 
by putting $V_{\mathfrak{u}}=A\otimes_{A_+}\K_\zeta$ where $A_+$ is
the subalgebra of $A$ generated over $\K$ by $X_1,\ldots,X_n$,
and $\K_\zeta$ is the $1$-dimensional module over $A_+$ given by $X_i.1:= \zeta_i$.
The map $\iota:R\to V_{\mathfrak{u}}$, $r\mapsto r\otimes 1$ is an $R$-module isomorphism.
Then there is a unique morphism of Whittaker pairs from $(V_{\mathfrak{u}}, 1\otimes 1)$
to any other Whittaker pair $(V,v_0)$ of type $\zeta$. And, identifying $V_{\mathfrak{u}}$ with $R$ via $\iota$,
 the kernel of the map $V_{\mathfrak{u}}\to V$, is precisely $\Ann_R v_0$.
 So if $(V,v_0)$ and $(W,w_0)$ are two Whittaker pairs
with $\Ann_R v_0= \Ann_R w_0$, it means that they are isomorphic to the
same quotient of the universal Whittaker pair of type $\zeta$, hence
are isomorphic to eachother.

c) If $\varphi: (V,v_0)\to (W,w_0)$ is a morphism of Whittaker pairs, then
$\varphi(rv_0)=r\varphi(v_0)=rw_0$ so clearly $\Ann_R v_0 \subseteq \Ann_R w_0$.
Conversely, if $Q_1\subseteq Q_2$ are proper $\Z^n$-invariant left ideals,
then there is an $R$-module morphism $\pi:R/Q_1\to R/Q_2$ mapping $1+Q_1$ to $1+Q_2$.
Since $\pi$ commutes with the $\Z^n$-action, one verifies that $\pi$
is automatically an $A$-module morphism.
\end{proof}

\begin{cor}
Let $A=\MTWA{n}{k}{r}{s}{\Lambda}/\langle \nf\rangle$ be a simple quotient of a
multiparameter twisted Weyl algebra as obtained in Theorem \ref{thm:family_of_simples}.
Then $A$ has a Whittaker module iff
\begin{equation}\label{eq:mtwa_whittaker_cond}
\la_{ij}=(r_{ij}/r_{ji})^k\quad\forall i,j.
\end{equation}
Moreover, if \eqref{eq:mtwa_whittaker_cond} holds, then for each
$\zeta\in(\K\backslash\{0\})^n$ there is a unique Whittaker module
over $A$ of type $\zeta$, namely the universal one, and it is a simple module.
\end{cor}

\end{document}